\documentclass[english]{article}
\title{A Solution Theory for a General Class of SPDEs}
\author{Andr\'e S{\"u}\ss{}\footnote{unaffiliated}\phantom{aa}\!\!and Marcus Waurick\footnote{Corresponding author, University of Bath, Claverton Down, Department of Mathematical Sciences, BA2 7AY, Bath, UK, m.waurick@bath.ac.uk, +44-1225-38-5483.}}
\date{Version of \today}

\usepackage[utf8]{inputenc}
\usepackage{mathtools}
\usepackage{babel}
\usepackage{url}
\usepackage{a4wide}
\usepackage{xcolor}
\usepackage{amsmath,amsthm,amssymb,amsxtra}
\usepackage{mathrsfs}
\usepackage{enumitem}

\let\geq\geqslant
\let\leq\leqslant

\allowdisplaybreaks

\newcommand*{\scrA}{\ensuremath{\mathscr{A}}}
\newcommand*{\scrF}{\ensuremath{\mathscr{F}}} 

\newcommand*{\N}{\mathbb{N}}									
\newcommand*{\R}{\mathbb{R}}									
\newcommand*{\C}{\mathbb{C}}									
\newcommand*{\Z}{\mathbb{Z}}									
\newcommand*{\bfS}{\mathbf{S}}									
\newcommand*{\Rd}{{\mathbb{R}^d}}							

\newcommand*{\1}{\chi}									

\newcommand*{\eps}{\varepsilon}								
\newcommand*{\ii}{\mathrm{i}}

\newcommand*{\E}{\mathbb{E}}									
\renewcommand*{\P}{\mathbb{P}}								

\DeclareMathAccent{\Circ}{\mathalpha}{operators}{"17}
\newcommand{\interior}[1]{\Circ{#1}}
\DeclareMathOperator*{\dom}{dom}
\DeclareMathOperator*{\ran}{ran}
\DeclareMathOperator*{\kar}{ker}
\DeclareMathOperator{\lin}{lin}
\DeclareMathOperator{\spt}{spt}
\renewcommand*{\div}{\mathrm{div}}
\DeclareMathOperator{\grad}{grad}
\DeclareMathOperator{\curl}{curl}

\numberwithin{equation}{section}

\theoremstyle{plain}
\newtheorem{lem}{Lemma}[section]
\newtheorem{thm}[lem]{Theorem}
\newtheorem{prop}[lem]{Proposition}
\newtheorem{cor}[lem]{Corollary}

\theoremstyle{definition}
\newtheorem{defn}[lem]{Definition}

\theoremstyle{remark}
\newtheorem{rem}[lem]{Remark}
\newtheorem{ex}[lem]{Example}

\definecolor{rot}{rgb}{1,0,0}

\begin{document}
\maketitle

\begin{abstract}
	In this article we present a way of treating stochastic partial differential equations with multiplicative noise by rewriting them as stochastically perturbed evolutionary equations in the sense of \cite{picardbook}, where a general solution theory for deterministic evolutionary equations has been developed. This allows us to present a unified solution theory for a general class of SPDEs which we believe has great potential for further generalizations. We will show that many standard stochastic PDEs fit into this class as well as many other SPDEs such as the stochastic Maxwell equation and time-fractional stochastic PDEs with multiplicative noise on sub-domains of $\Rd$. The approach is in spirit similar to the approach in \cite{dapratozabczyk}, but complementing it in the sense that it does not involve semi-group theory and allows for an effective treatment of coupled systems of SPDEs. In particular, the existence of a (regular) fundamental solution or Green's function is not required.
\end{abstract}


\vspace*{0.125cm}


\section*{Acknowledgments}

The authors wish to thank Markus Kunze for very useful correspondence. A.~S.~was supported by the grant MTM 2012-31192 from the Direcci\'on General de Investigaci\'on Cient\'ifica y T\'ecnica, Ministerio de Econom\'ia y Competitividad, Spain. M.~W.~carried out this work with financial support of the EPSRC grant EP/L018802/2: ``Mathematical foundations of metamaterials: homogenisation, dissipation and operator theory''. This is gratefully acknowledged. We thank an anonymous referee for helpful comments und remarks, which helped to improve the manuscript in many respects.

\section{Introduction}
The study of stochastic partial differential equations (SPDEs) attracted a lot of interest in the recent years, with a wide range of equations already been investigated. A common theme in the study of these equations is to attack the problem of existence and uniqueness of solutions to SPDEs by taking solution approaches from the deterministic setting of PDEs and applying them to a setting that involves a stochastic perturbation. Examples for this are the random-field approach that uses the fundamental solution to the associated PDE in \cite{walsh,dalang,conusdalang}, the semi-group approach which treats evolution equations in Hilbert/Banach spaces via the semi-group generated by the differential operator of the associated PDE, see \cite{dapratozabczyk} or \cite{Kunze2011,Kunze2012} for a treatise, and the variational approach which involves evaluating the SPDE against test functions, which corresponds to the concept of weak solutions of PDEs, see \cite{sanzvuillermot,rozowskii,prevotroeckner}. 

In this article we aim to transfer yet another solution concept of PDEs to the case when the right-hand side of the PDE is perturbed by a stochastic noise term. This solution concept, see \cite{picardbook} for a comprehensive study and \cite{Waurick2014MMAS_Non,PTWW,Trostorff_nonl} for possible generalizations, is of operator-theoretic nature and takes place in an abstract Hilbert space setting. Its key features are establishing the time-derivative operator as a normal, continuously invertible operator on an appropriate Hilbert space and a positive definiteness constraint on the partial differential operator of the PDE (realized as an operator in space-time). Actually, this solution theory is a general recipe to solve a first-order (in time and space) system of coupled equations, and when solving a higher-order (S)PDE, it gets reduced to such a first-order system. In this sense the solution theory we will apply is roughly similar in spirit to the treatment of hyperbolic equations in \cite{hormander, kumano-go}, see also \cite{ascanellisuess} for an application to SPDEs. 

We shall illustrate the class of SPDEs we will investigate using this approach. Throughout this article let $H$ be a Hilbert space, that we think of as basis space for our investigation. We assume $A$ to be a skew-self-adjoint, unbounded linear operator on $H$ (i.e.\ $\ii A$ is a self-adjoint operator on $H$) which is thought of as containing the spatial derivatives. Furthermore, we denote by $\partial_0$ the time-derivative operator that will be constructed as a normal and continuously invertible operator in Section \ref{sec:partial0}. In particular, it can be shown that the spectrum of $\partial_0^{-1}$ is contained in a ball of the right half plane touching $0\in \mathbb{C}$.  Let for some $r>0$, $M:B(r,r)\to L(H)$ be an analytic function, where $B(r,r)$ is the open ball in $\mathbb{C}$ with radius $r>0$ centered at $r>0$, and $L(H)$ the set of bounded linear operators on $H$. Then one can define via a functional calculus the linear operator $M(\partial_0^{-1})$ as a function of the inverse operator $\partial_0^{-1}$, which will be specified below. The idea to define this operator is to use the Fourier--Laplace transform as explicit spectral 
representation as multiplication operator for $\partial_0$ yielding a functional calculus for both $\partial_0$ and its inverse. The role that $M(\partial_0^{-1})$ plays is coupling the equations in the first-order system. In applications, $M(\partial_0^{-1})$ also contains the information about the `constitutive relations' or the `material law'.

Throughout this article we consider the following (formal) system of coupled SPDEs
\begin{equation}\label{eq:SPDE}
	\big(\partial_0M(\partial_0^{-1})+A\big)u(t) = (B(u))(t)+\int_0^t \sigma(u(s))dW(s),
\end{equation}
subject to suitable initial conditions, where $u(t)$ admits values in a Hilbert space $H$, $\sigma$ and $B$ are Lipschitz-continuous (in some suitable norms) functions and $W$ is a cylindrical $G$-valued Wiener process for some separable Hilbert space $G$ (possibly different from $H$). Though seeming to represent first-order equations, only, it is possible to handle for instance the wave (or heat) equation with \eqref{eq:SPDE} as well, see below. Moreover, note that $M(\partial_0^{-1})$ is an operator acting in space-time and $\partial_0M(\partial_0^{-1})$ is the composition of time differentiation and the application of the operator $M(\partial_0^{-1})$.

We emphasize that in the formulation of \eqref{eq:SPDE}, $A$ does not admit the usual form of stochastic evolution equations as, for instance, in \cite{dapratozabczyk}. Furthermore, \eqref{eq:SPDE} should not be thought of being of a similar structure as the problems discussed in \cite{Gerencser2015,Lototsky2013}. In fact, the coercitivity is encoded in $M(\partial_0^{-1})$ rather than $A$.

In equation \eqref{eq:SPDE} possible boundary conditions are encoded in the domain of the (partial differential) operator $A$. The way of dealing with this issue will also be further specified below. The main achievement of this article is the development of a suitable functional analytic setting for the class of equations \eqref{eq:SPDE}, which allows us to discuss well-posedness issues of this class of equations, that is, existence, uniqueness and continuous dependence of solutions on the input data. 

Now we comment on a notational unfamiliarity in equation \eqref{eq:SPDE}. Note that in \eqref{eq:SPDE} a stochastic integral appears on the right-hand side of the equation instead of the more familiar formal product $\sigma(u(t))\dot W(t)$. We stress here that we are \emph{not} aiming at solving a different class of equations, but in fact we deal with a more general formulation of the common way to write an SPDE. Let us illustrate this point using two common examples, the stochastic heat equation and the stochastic wave equation. The former is usually expressed in the classic formulation in the following way
\[
   d u(t)= (\Delta + b(u(t)))dt+\sigma(u(t))d W(t),
\]
or -- formally dividing by $dt$ --
\[ \bigg(\frac{\partial}{\partial t}-\Delta\bigg)u(t) = b(u(t))+\sigma(u(t))\dot W(t), \]
where $\Delta=\div\grad$ is the Laplacian on the Euclidean space $\R^d$ with $d\in\N$, and $b,\sigma$ are linear or nonlinear mappings on some Hilbert space, for instance some $L^2$-space over $\R^d$. See \cite[Chapter 7]{dapratozabczyk} for more details on this formulation. We can reformulate this equation as a first-order system using the formal definition $v\coloneqq -\grad\partial_0^{-1}u$, where $\partial_0^{-1}$ denotes the inverse of the time derivative operator briefly mentioned above. Then the stochastic heat equation becomes
\[\left(\partial_0\begin{pmatrix} 0 & 0 \\ 0 & 1\end{pmatrix} + \begin{pmatrix} 1 & 0 \\ 0 & 0\end{pmatrix} + \begin{pmatrix} 0 & \div \\ \grad & 0\end{pmatrix} \right)\begin{pmatrix} u \\ v\end{pmatrix} = \begin{pmatrix} \partial_0^{-1}b(u) + \partial_0^{-1}(\sigma(u)\dot W) \\ 0\end{pmatrix}. \]
Thus, with $(\partial_0^{-1}b(u),0)=B(u)$ and if we interpret the term  $\partial_0^{-1}(\sigma(u)\dot W)$ as a stochastic integral, we immediately arrive at \eqref{eq:SPDE}. So the operator-valued function $M$ and the operator $A$ in \eqref{eq:SPDE} respectively equal
\[ M(z) = \begin{pmatrix} z & 0 \\ 0 & 1\end{pmatrix}\text{ and }A = \begin{pmatrix} 0 & \div \\ \grad & 0\end{pmatrix}. \]
Indeed, with these settings, we get
\[
 \partial_0M(\partial_0^{-1})+A = \partial_0\begin{pmatrix} \partial_0^{-1} & 0 \\ 0 & 1\end{pmatrix}+\begin{pmatrix} 0 & \div \\ \grad & 0\end{pmatrix}= \partial_0\begin{pmatrix} 0 & 0 \\ 0 & 1\end{pmatrix}+\begin{pmatrix} 1 & 0 \\ 0 & 0\end{pmatrix}+\begin{pmatrix} 0 & \div \\ \grad & 0\end{pmatrix}.
\]
In particular, $H=L^2(\lambda_{\mathbb{R}^d} )^{d+1}$, where $\lambda_{\mathbb{R}^d} $ denotes the Lebesgue measure on $\mathbb{R}^d$.

In a similar fashion, one can reformulate the stochastic wave equation, which in the classic formulation is given by 
\[ \bigg(\frac{\partial^2}{\partial t^2}-\Delta\bigg)u(t) = b(u(t))+\sigma(u(t))\dot W(t), \]
by using $v=-\grad \partial_0^{-1}u$ as
\[ \left(\partial_0\begin{pmatrix} 1 & 0 \\ 0 & 1\end{pmatrix} + \begin{pmatrix} 0 & \div \\ \grad & 0\end{pmatrix} \right)\begin{pmatrix} u \\ v\end{pmatrix} = \begin{pmatrix} \partial_0^{-1}b(u) + \partial_0^{-1}(\sigma(u)\dot W) \\ 0\end{pmatrix}, \]
where here
\[ M(z) = \begin{pmatrix} 1 & 0 \\ 0 & 1\end{pmatrix}. \]
In comparison to the example of the stochastic heat equation, this formulation in terms of a first-order system is already well-known and heavily used. The main advantage of the formulation \eqref{eq:SPDE} is that many more examples of PDEs in mathematical physics can be written in this form, see \cite{PicPhy}. The hand-waving arguments handling $\partial_0^{-1}$ that we have used in the reduction to first-order systems will be made rigorous in Section \ref{sec:deterministic}.

This paper is structured in the following way. In Section \ref{sec:deterministic} we present a brief overview over the solution theory for PDEs which will be used in this article, in particular we explain the construction of the time derivative operator and the concept of so-called Sobolev chains. 
We state the results and sketch the respective proofs referring to \cite{picardbook} for the details and highlight some further generalizations. 
In the subsequent Section \ref{sec:stochastics} we show how the solution theory for deterministic PDEs carries over to the case of SPDEs which we think of as random perturbations of PDEs. We clarify the way how to interpret the stochastic integral, and then present a solution theory to SPDEs with additive and multiplicative noise. 
In Section \ref{sec:examples} we show using concrete examples how this solution theory can be successfully applied to concrete SPDEs, some of which -- to the best of our knowledge -- have not yet been solved in this level of generality.  We conclude Section \ref{sec:examples} with a SPDE of mixed type, that is, an equation which is hyperbolic, parabolic and elliptic on different space-time regions. This demonstrates the versatility of the approach presented as for instance the semi-group method fails to work in this example for there is no semi-group to formulate the (non-homogeneous) Cauchy problem in the first place.
Furthermore, we provide some connections of this new solution concept to some already known approaches to solve SPDEs. More precisely, we draw the connection of variational solutions of the heat equation to the solutions obtained here. Further, for the stochastic wave equation, we show that the mild solution derived via the semi-group method coincides with the solution constructed in this exposition. We summarize our findings in Section \ref{s:con}.

In this article we denote the identity operator by $1$ or by $1_H$ and indicator functions by $\chi_K$ for some set $K$. The Lebesgue measure on a measurable subset ${D} \subseteq \mathbb{R}^d$ for some $d\in \mathbb{N}$ will be denoted by $\lambda_{D}$. All Hilbert spaces in this article are endowed with $\mathbb{C}$ as underlying scalar field. $L^2$-spaces of (equivalence classes of) scalar-valued square integrable functions over a measure space $(\Omega,\mathcal{A},\mu)$ are denoted by $L^2(\mu)$. The corresponding space of Hilbert space $H$-valued $L^2$-functions will be denoted by $L^2(\mu;H)$. $\mathbb{P}$ will always denote a probability measure.

\section{The deterministic solution theory}\label{sec:deterministic}
In this section we will review the solution theory for a class of linear partial differential equations developed in \cite[Chapter 6]{picardbook} or \cite{PicPhy}.  This solution theory of partial differential equations relies on 2 main observations: (1) to establish the time-derivative operator as a normal and continuously invertible operator on an appropriate Hilbert space and (2) a positive definiteness constraint on the partial differential operators realized as operators in space-time.

\subsection{Functional analytic ingredients}\label{sec:partial0}

Let throughout this article $\nu>0$. This is a free parameter which controls the growth of solutions to PDEs for large times. Consider the space
\[
   H_{\nu,0}(\R)\coloneqq \{ f\in L^2_{\text{loc}}(\lambda_\R); \left(x\mapsto e^{-\nu x}f(x) \right)\in L^2(\lambda_\R) \}
\]
of $L^2$-functions with respect to the exponentially weighted Lebesgue measure $\exp(-2\nu (\cdot))\lambda_\R$. The latter space becomes a Hilbert space if endowed with the scalar product
\begin{align*}
  \langle\cdot,\cdot\rangle_{\nu,0} \colon H_{\nu,0}(\R) \times  H_{\nu,0}(\R) &\to \C,\\
(f,g) &\mapsto \int_{\R} f(x)^*g(x)e^{-2\nu x} d x,
\end{align*}
where $^*$ denotes complex conjugation. Note that the operator $\exp(-\nu m)$ given by
\begin{equation}\label{e:exp}
  \exp(-\nu m)\colon H_{\nu,0}(\R)\to L^2(\lambda_\R),f\mapsto e^{-\nu \cdot}f(\cdot)
\end{equation}
of multiplying with the function $t\mapsto e^{-\nu t}$ is unitary from $H_{\nu,0}(\mathbb{R})$ to $H_{0,0}(\mathbb{R})(=L^2(\lambda_\mathbb{R}))$.

Define $\dom(\partial_{0,\nu})\coloneqq \{ f\in H_{\nu,0}(\R); f'\in H_{\nu,0}(\R)\}$, where $f'$ is the distributional derivative of $f\in L^1_{\text{loc}}(\lambda_\R)$, and
\begin{equation}\label{e:der}
   \partial_{0,\nu} \colon \dom(\partial_{0,\nu}) \subseteq H_{\nu,0}(\R) \to H_{\nu,0}(\R), f\mapsto f'.
\end{equation}
Then this operator has the following properties, see also \cite[Corollary 2.5]{Kalauch}.

\begin{lem}\label{lem:2.1}
  $\partial_{0,\nu}$ is a continuously invertible linear operator with $\lVert \partial_{0,\nu}^{-1}\rVert\leq \frac{1}{\nu}$ and $\Re \partial_{0,\nu}=\nu$.
\begin{proof}
Recall $\exp(-\nu m)$ from \eqref{e:exp} is unitary. By the product rule we deduce the equality
\begin{equation}\label{eq:pdnu}
   \partial_{0,\nu} = \exp(-\nu m)^{-1}(\partial +\nu)\exp(-\nu m),
\end{equation}
where $\partial\colon H^1(\R)\subseteq L^2(\lambda_\R)\to L^2(\lambda_\R)$ is the (usual) distributional derivative operator realized in $L^2(\lambda_\R)$.
Indeed, for a smooth compactly supported function $\phi$, we observe that
\begin{align*}
   &\left(\exp(-\nu m)^{-1}(\partial +\nu)\exp(-\nu m)\phi\right)(x)
   \\& =\exp(\nu x)\left((\partial +\nu)\exp(-\nu m)\phi\right)(x)
   \\& =\exp(\nu x)(-\nu e^{-\nu x}\phi(x)+e^{-\nu x}\phi'(x) +\nu\exp(-\nu x)\phi(x)) = \phi'(x) = \partial_{0,\nu}\phi(x)
\end{align*}
Since $\partial$ is skew-self-adjoint in $L^2(\lambda_\R)$ (\cite[Chapter V, Example 3.14]{Kato1980}), the spectrum of $\partial$ lies on the imaginary axis. Hence, the operator $\partial+\nu$ is continuously invertible. By \eqref{eq:pdnu}, the operators $\partial+\nu$ and $\partial_{0,\nu}$ are unitarily equivalent. Thus, the operator $\partial_{0,\nu}$ is continuously invertible as well. The norm estimate also follows from \eqref{eq:pdnu} as so does the formula $\Re\partial_{0,\nu}=\nu$ since $\Im(\partial_{0,\nu})=\exp(-\nu m)^{-1}((-i)\partial)\exp(-\nu m)$, by the skew-self-adjointness of $\partial$.
\end{proof}
\end{lem}

\begin{rem}\label{r:pif} By \cite[Corollary 2.5 (d)]{Kalauch}, we have
\[
   \partial_{0,\nu}^{-1}f(t)=\int_{-\infty}^t f(\tau)d \tau \quad (t\in \mathbb{R})
\]
for all $f\in H_{\nu,0}(\mathbb{R})$.
\end{rem}

Note that for a Hilbert space $H$, there exists a canonical extension of $\partial_{0,\nu}$ to the space $H_{\nu,0}(\R;H)$ of corresponding $H$-valued functions by identifying $H_{\nu,0}(\R;H)$ with $H_{\nu,0}(\R)\otimes H$ and the extension of $\partial_{0,\nu}$ by $\partial_{0,\nu}\otimes 1_H$.

An important tool in this article is the (Hilbert space valued) Fourier transformation \[\mathcal{F}\colon L^2(\lambda_\R;H)\to L^2(\lambda_\R;H)\] defined by the unitary extension of
\[
   \mathcal{F}\phi(x) \coloneqq \frac{1}{\sqrt{2\pi}}\int_\R e^{-\ii xy}\phi(y)dy\quad (x\in \R,\phi\in L^1(\lambda_\R;H)\cap L^2(\lambda_\R;H)),
\]
to $L^2(\lambda_\R;H)$. In fact, the norm preservation is the same as saying that Plancherel's theorem also holds for the Hilbert space valued case. Recall that the inverse Fourier transform satisfies $(\mathcal{F}^{-1}\phi)(x)=(\mathcal{F}^*\phi)(x) = (\mathcal{F}\phi)(-x)$. 

Next, recall (\cite[Volume 1, p.161-163]{Akhiezer1981}) that for the derivative $\partial\colon H^1(\R)\subseteq L^2(\lambda_\R)\to L^2(\lambda_\R)$, the Fourier transformation realizes an explicit spectral representation for $\partial$ as multiplication operator in the Fourier space:
\[
    \partial = \mathcal{F}^* \ii m \mathcal{F},
\]
where $(mf)(x)\coloneqq xf(x)$ denotes the multiplication-by-argument-operator in $L^2(\lambda_\R;H)$. 

We define the \emph{Fourier--Laplace transformation} $\mathcal{L}_\nu \coloneqq \mathcal{F}\exp(-\nu m)$ with $\exp(-\nu m)$ given in \eqref{e:exp}. Then, $\mathcal{L}_\nu$ defines a spectral representation for $\partial_{0,\nu}$ given in \eqref{e:der} (and hence also for $\partial_{0,\nu}^{-1}$). Indeed, we get $\partial_{0,\nu} = \mathcal{L}_\nu^* (\ii m+\nu)\mathcal{L}_\nu$ and
\[
    \partial_{0,\nu}^{-1} = \mathcal{L}_\nu^* \left( \frac{1}{\ii m+\nu}\right)\mathcal{L}_\nu.
\]
The latter formula carries over to \emph{(operator-valued)-functions} of $\partial_{0,\nu}^{-1}$, that is, we set up a functional calculus for $\partial_{0,\nu}^{-1}$. We define
\begin{equation}\label{eq:defM}
  M(\partial_{0,\nu}^{-1})\coloneqq \mathcal{L}_\nu^* M\left( \frac{1}{\ii m+\nu}\right)\mathcal{L}_\nu,
\end{equation}
where $M\colon B(r,r)\to L(H)$ is analytic and bounded, $r>\frac{1}{2\nu}$, as well as for all $x\in \R$ and $\phi\in C_c(\R;H)$
\[
   M\left( \frac{1}{\ii m+\nu}\right)\phi(x)\coloneqq M\left( \frac{1}{ \ii x+\nu}\right)\phi(x).
\]
Note that the right hand side is the application of the bounded linear operator $M\left( \frac{1}{\ii x+\nu}\right)\in L(H)$ to the Hilbert space element $\phi(x)\in H$.

In principle, one could cope with (operator-valued) functions $M$ being defined on $\partial B(r',r')\setminus\{0\}$ with $r'\coloneqq 1/(2\nu)$, only. In fact, \eqref{eq:defM} is still possible. However, in the solution theory to be developed in the next section, we want to establish causality for the solution operator, that is, the solution vanishes up to time $t$ if the data do (see below for the details). But, vanishing up to time $0$ is intimately related to analyticity:

We denote the open complex right half plane by $\C_{>0}=\{ \ii t+\nu; t\in \mathbb{R}, \nu>0\}$.

\begin{thm}[{{Paley--Wiener, cf.\ \cite[Chapter 19]{Rudin1987} and \cite[Corollary 2.7]{PicPhy}}}]\label{thm:PW} Let $H$ be a Hilbert space, $u\in L^2(\lambda_\mathbb{R};H)$. Then the following properties are equivalent:
\begin{enumerate}
 \item $\C_{>0} \ni \ii t+\nu \mapsto (\mathcal{L}_{\nu} u)(t) \in H$ belongs to the \emph{Hardy--Lebesgue} space 
\begin{multline*}
    \mathcal{H}^2(H)\coloneqq \{ f\colon \C_{>0}\to H; f\text{ analytic,}\\ f(\ii \cdot+\nu)\in L^2(\lambda_\mathbb{R};H)\, (\nu>0), \sup_{\nu>0}\|f(\ii \cdot +\nu)\|_{L^2}<\infty\}                                                                                                            
                                                                                                           \end{multline*}
 \item $u=0$ on $(-\infty,0)$.                                                                                                
\end{enumerate} 
\end{thm}

We introduce Sobolev chains, which may be needed in the later investigation, see \cite[Chapter 2]{picardbook}, or \cite{PicSoLa}. These concepts are the natural generalizations of Gelfand triples to an infinite chain of rigged Hilbert spaces. We shall also refer to similar concepts developed in \cite{KPY,GelVile} or, more recently, \cite{EngNag}.

\begin{defn}
  Let $C:\dom(C)\subseteq H\to H$ be densely defined and closed. If $C$ is continuously invertible, then we define $H_k(C)$ to be the completion of $(\dom(C^{|k|}),\lVert C^{k}\cdot \rVert_{H})$ for all $k\in\Z$. The sequence $(H_k(C))_k$ is called \emph{Sobolev chain} associated with $C$.
\end{defn}

Obviously, $H_k(C)$ is a Hilbert space for each $k\in\Z$. Moreover, it is possible to extend the operator $C$ unitarily to an operator from $H_k(C)$ to $H_{k-1}(C)$. We will use these extensions throughout and use the same notation. It can be shown that $H_{k}(C^*)^*$ can be identified with $H_{-k}(C)$ via the dual pairing
\[
   H_k(C^*)\times H_{-k}(C) \ni (\phi,\psi)\mapsto \left\langle \left(C^*\right)^k\phi,C^{-k}\psi \right\rangle_H
\]
for all $k\in\Z$, where we identify $H$ with its dual space. Further, note that $H_k\hookrightarrow H_m$ as long as $k\geq m$. Hence, the name ``chain''.

\begin{ex}
(a) A particular example for such operators $C$ is the time-derivative $\partial_{0,\nu}$. We denote $H_{\nu,k}(\R)\coloneqq H_k(\partial_{0,\nu})$ for all $k\in\Z$ and correspondingly for the Hilbert-space-valued case. 

(b) A second important example to be used later on is the case of a skew-self-adjoint operator $A$ in some Hilbert space $H$. We build the Sobolev chain associated with $C=A+1$.
\end{ex}

\subsection{The solution theory}\label{sec:soltheory}

The solution theory which we will apply covers a large class of partial differential equations in mathematical physics. We will summarize it in this section, and for convenience, we shall also provide outlines of the proofs. For the whole arguments, the reader is referred to \cite{picardbook} and \cite{PTWW,Waurick2014MMAS_Non}. The following observation, a variant of coercitivity, provides the functional analytic foundation. 

\begin{lem}\label{lem:pos_def}
Let $G$ be a Hilbert space, $B\colon \dom(B)\subseteq G\to G$ a densely defined, closed, linear operator. Assume there exists $c>0$ with the property that 
\begin{equation}\label{eq:coerc1}
	\Re \langle B\phi,\phi\rangle\geq c\langle \phi,\phi\rangle,
\end{equation}
and
\begin{equation}\label{eq:coerc2}
	\Re \langle B^*\psi,\psi\rangle\geq c\langle \psi,\psi\rangle,
\end{equation}
for all $\phi\in \dom(B)$ and $\psi\in \dom(B^*)$. Then $B^{-1}$ exists as an element of $L(G)$, the space of bounded linear operators on $G$ and $\|B^{-1}\|\leq 1/c$. 
\end{lem}
\begin{proof}
Using the Cauchy-Schwarz-inequality, we can read off from the first inequality \eqref{eq:coerc1} that $B$ is one-to-one. More precisely, we have for all $\phi\in \dom(B)$
 \begin{equation}\label{eq:inequality}
    c\|\phi\|\leq \|B\phi\|.
 \end{equation}
Thus, $B^{-1}$ is well-defined on $\ran(B)$, the latter being a closed subset of $G$. In fact, take $(\psi_n)_n$ in $\ran(B)$ converging to some $\psi\in G$. We find $(\phi_n)_n$ in $\dom(B)$ with $B\phi_n=\psi_n$. Then, again relying on the inequality involving $B$, we get
 \[
    c\| \phi_n -\phi_m\|\leq \| B\phi_n - B\phi_m\|=\| \psi_n - \psi_m\|\quad (n,m\in \N),
 \]
 which shows that $(\phi_n)_n$ is a Cauchy-sequence in $G$, and, thus, convergent to some $\phi\in G$. The closedness of $B$ gives that $\phi\in \dom(B)$ and $B\phi=\psi\in \ran(B)$ as desired. 
 
Next, again by the Cauchy-Schwarz inequality, we deduce that also $B^*$ is one-to-one, or expressed differently $\kar(B^*)=\{0\}$. Thus, by the projection theorem, $G=\kar(B^*)\oplus \overline{\ran(B)}=\{0\}\oplus \ran(B)$ yielding that $B$ is onto. The inequality for the norm of $B^{-1}$ can be read off from \eqref{eq:inequality} by setting $\phi := B^{-1}g$ for any $g\in \ran(B)=G$:
\[ c\|B^{-1}g\| \leq \|BB^{-1}g\| = \|g\|. \]
This finishes the proof.
\end{proof}

\begin{rem}\label{rem:canext}
Given a densely defined closed linear operator $A_0:\dom(A_0)\subseteq H\to H$, there exists a closed, densely defined (canonical) extension $A$ to $H_{\nu,0}(\R;H)$ in the way that $(Au)(t)\coloneqq A_0u(t)$ for $t\in\R$ and $u\in C_c (\R;\dom(A_0))$. Indeed, the construction can be done similarly to the extension of the time-derivative by setting $A\coloneqq 1_{H_{\nu,0}(\R)}\otimes A_0$. Then, if $A_0$ is continuously invertible, then so is $A$. The adjoint of $A$ is the extension of the adjoint of $A_0$. Due to these similarities there is little use in distinguishing notationally $A_0$ from its extension $A$. Hence, we will use throughout the same notation for $A_0$ and its extension.
\end{rem} 

The next result is the main existence and uniqueness theorem in the deterministic setting.

\begin{thm}[{\cite[Theorem 6.2.5]{picardbook}, \cite[Solution Theory]{PicPhy}}]\label{thm:solth}
Let $H$ be a Hilbert space, $A:\dom(A)\subseteq H\to H$ a skew-self-adjoint linear operator. For some $r>0$, let $M\colon B(r,r)\to L(H)$ be a bounded and analytic mapping. Assume that there exists $c>0$ such that
\begin{equation}\label{eq:PDofM}
  \Re \langle z^{-1} M(z )\phi,\phi\rangle_H  \geq c \|\phi\|_H^2 \quad (z\in B(r,r),\, \phi\in H).
\end{equation}
Then for all $\nu> 1/(2r)$ the operator
\[
   \partial_{0,\nu}M(\partial_{0,\nu}^{-1})+A \colon \dom(A)\cap \dom(\partial_{0,\nu}) \subseteq  H_{\nu,0}(\R;H) \to H_{\nu,0}(\R;H)
\]
is closable with continuously invertible closure. Denoting $S_\nu$ to be the inverse of the closure, 
\[ S_\nu := \big(\overline{\partial_{0,\nu}M(\partial_{0,\nu}^{-1})+A}\big)^{-1}, \]
we get that $S_\nu$ is \emph{causal}, that is, for $a\in \R$ and $f,g\in H_{\nu,0}(\R;H)$ the implication
\begin{equation}\label{eq:defa}
  f=g\text{ on }(-\infty,a) \Rightarrow S_\nu f = S_\nu g \text{ on }(-\infty,a)
\end{equation}
holds true. Moreover, $\|S_\nu\|\leq c^{-1}$.
\end{thm}
\begin{proof} 
At first we show the existence and uniqueness of solutions to the equation
\[
   \overline{(\partial_{0,\nu}M(\partial_{0,\nu}^{-1})+A)}u = f
\]
for given $f\in H_{\nu,0}(\mathbb{R};H)$, which boils down to (closability and) continuous invertibility of the (closure of the) partial differential operator $B_0 \coloneqq \partial_{0,\nu}M(\partial_{0,\nu}^{-1})+A$. 

At first, we observe that $B_0$ with $\dom(B_0)=\dom(A)\cap \dom(\partial_{0,\nu})$ is closable. Indeed, it is easy to check that $\partial_{0,\nu}^*M(\partial_{0,\nu}^{-1})^*-A$ with dense domain $\dom(B_0)$ is a formal adjoint. Hence, $B_0$ is closable. For the proof of the continuous invertibility of $B\coloneqq \overline{B_0}$, we apply Lemma \ref{lem:pos_def} with $G\coloneqq H_{\nu,0}(\R;H)$. To this end, take $\phi\in \dom(B_0)$ and compute
\begin{align*}
  \Re \langle (\partial_{0,\nu}M(\partial_{0,\nu}^{-1})+A)\phi,\phi\rangle & = \Re \langle \partial_{0,\nu}M(\partial_{0,\nu}^{-1})\phi,\phi\rangle+ \Re\langle A\phi,\phi\rangle \\
  & =\Re \langle \partial_{0,\nu}M(\partial_{0,\nu}^{-1})\phi,\phi\rangle\\
  & =\Re \langle (\ii m+\nu)M\left(\frac{1}{\ii m+\nu}\right)\mathcal{L}_\nu\phi,\mathcal{L}_\nu\phi\rangle\\
  & \geq c \langle \mathcal{L}_\nu\phi,\mathcal{L}_\nu\phi\rangle= c \langle \phi,\phi\rangle,  
\end{align*}
where we have used that $A$ is skew-self-adjoint (hence, $\Re \langle A\phi,\phi\rangle=-\Re \langle \phi,A\phi\rangle=-\Re \langle A\phi,\phi\rangle$), \eqref{eq:defM}, \eqref{eq:PDofM} as well as Plancherel's identity, that is, the unitarity of $\mathcal{L}_\nu$. This inequality carries over to all $\phi\in \dom(B)$. 

In order to use Lemma \ref{lem:pos_def}, we need to compute the adjoint of $B$. For this, note that $(1+\epsilon\partial_{0,\nu}^*)^{-1}$ converges strongly to the identity as $\epsilon\to 0$. So, fix $f\in \dom(B^*)$ and $\epsilon>0$. Observe that $(1+\epsilon\partial_{0,\nu})^{-1}$ commutes with $B_0$ and leaves the space $\dom(B_0)$ invariant. Then we compute for $\phi\in \dom(B_0)$
\begin{align*}
 \langle \phi,(1+\epsilon\partial_{0,\nu}^*)^{-1}B^*f\rangle & = \langle (1+\epsilon\partial_{0,\nu})^{-1}\phi,B^*f\rangle
 \\ & = \langle B_0 (1+\epsilon\partial_{0,\nu})^{-1}\phi,f\rangle
 \\ & = \langle (1+\epsilon\partial_{0,\nu})^{-1}B_0\phi,f\rangle
 \\ & = \langle B_0 \phi,(1+\epsilon\partial_{0,\nu}^*)^{-1}f\rangle
 \\ & = \langle \partial_{0,\nu}M(\partial_{0,\nu}^{-1}) \phi,(1+\epsilon\partial_{0,\nu}^*)^{-1}f\rangle+ \langle A \phi,(1+\epsilon\partial_{0,\nu}^*)^{-1}f\rangle
 \\ & = \langle  \phi,M(\partial_{0,\nu}^{-1})^*\partial_{0,\nu}^*(1+\epsilon\partial_{0,\nu}^*)^{-1}f\rangle+ \langle A \phi,(1+\epsilon\partial_{0,\nu}^*)^{-1}f\rangle
 \\ & = \langle  \phi,\partial_{0,\nu}^*M(\partial_{0,\nu}^{-1})^*(1+\epsilon\partial_{0,\nu}^*)^{-1}f\rangle+ \langle A \phi,(1+\epsilon\partial_{0,\nu}^*)^{-1}f\rangle.
\end{align*}
Hence, as $\dom(B_0)$ is a core for $A$, we infer that $(1+\epsilon\partial_{0,\nu}^*)^{-1}f\in \dom(A^*)$ and that 
\[
  A^*(1+\epsilon\partial_{0,\nu}^*)^{-1}f=-A(1+\epsilon\partial_{0,\nu}^*)^{-1}f=(1+\epsilon\partial_{0,\nu}^*)^{-1}B^*f-\partial_{0,\nu}^*M(\partial_{0,\nu}^{-1})^*(1+\epsilon\partial_{0,\nu}^*)^{-1}f,
\]
or, equivalently, 
 \[
  (1+\epsilon\partial_{0,\nu}^*)^{-1}B^*f=(\partial_{0,\nu}^*M(\partial_{0,\nu}^{-1})^*-A)(1+\epsilon\partial_{0,\nu}^*)^{-1}f.
 \]
Note that also $\dom(B_0) =  \dom(\partial_{0,\nu}^*)\cap \dom(A)$, since $\dom(\partial_{0,\nu})=\dom(\partial_{0,\nu}^*)$. Letting $\epsilon\to 0$ in the last equality, we infer that
 \[
    B^* \subseteq \overline{(\partial_{0,\nu}^*M(\partial_{0,\nu}^{-1})^*-A)|_{\dom(B_0)}}.
 \]
But as $\Re \langle (\partial_{0,\nu}^*M(\partial_{0,\nu}^{-1})^*-A)\psi,\psi\rangle\geq c\langle\psi,\psi\rangle$ for all $\psi\in \dom(B_0)$, we conclude that for all $\psi\in \dom(B^*)$
 \[
   \Re \langle B^*\psi,\psi\rangle \geq c\langle \psi,\psi\rangle. 
 \]
Hence, Lemma \ref{lem:pos_def} implies that $B$ is continuously invertible, and we denote $S_\nu := B^{-1}$. The norm estimate for $\|S_{\nu}\|$ follows from Lemma \ref{lem:pos_def}.
 
The next step is to show causality, and here we only sketch the arguments and we refer to \cite[Section 2.2, Theorem 2.10]{PicPhy} for the details. First of all, note that $B$ commutes with time-translation $\tau_hf\coloneqq f(\cdot+h)$ as it is also a function of $\partial_{0,\nu}$. In fact, one has $\tau_h =\mathcal{L}_\nu^* e^{(\ii m +\nu)h}\mathcal{L}_\nu$. Hence, causality needs only being checked for $a=0$ in \eqref{eq:defa}. Moreover, by the linearity of $S_\nu$, it suffices to verify the implication in \eqref{eq:defa} for $g=0$.
So, take $f\in H_{\nu,0}(\R;H)$ vanishing on $(-\infty,0]$. We have to show that $S_\nu f$ also vanishes on $(-\infty,0]$. Observe that $e^{-\nu m}f\in L^2(\lambda_{[0,\infty)};H)$. Hence, by the Paley--Wiener theorem $\mathcal{L}_\nu f =\mathcal{F}e^{-\nu m} f$ belongs to the Hardy--Lebesgue space of analytic functions on the half plane being uniformly in $L^2(\lambda_\R;H)$ on any line parallel to the coordinate axis, see Theorem \ref{thm:PW}.

\noindent
Next, $\overline{((\ii m +\nu)M(\frac{1}{\ii m +\nu})+A)}^{-1}$ as multiplication operator on the Hardy--Lebesgue space leaves the Hardy-Lebesgue space invariant, by the boundedness of the inverse and the analyticity of both the resolvent map and the mapping $M$. Thus, $\overline{((\ii m +\nu)M(\frac{1}{\ii m +\nu})+A)}^{-1}\mathcal{L}_\nu f$ belongs to the Hardy--Lebesgue space. Thus, $\mathcal{F}^*\overline{((\ii m +\nu)M(\frac{1}{\ii m +\nu})+A)}^{-1}\mathcal{L}_\nu f$ is supported on $[0,\infty)$, by the Paley--Wiener theorem. Hence,
 \[
    S_\nu f = \mathcal{L}_\nu^*\overline{((\ii m +\nu)M(\frac{1}{\ii m +\nu})+A)}^{-1}\mathcal{L}_\nu f
 \]
is also supported on $[0,\infty)$ only, yielding the assertion.
\end{proof}

The operator $S_\nu$ defined in the previous theorem is also denoted as \emph{solution operator} to the PDE. The concept of causality is an action-reaction principle, i.e.\ only if there is some non-zero action on the right-hand side of the equation, the solution can become non-zero.

\begin{rem}\label{rem:integratedPDE} (a)
As it was pointed out in \cite[p.\ 494]{picardbook}, we can freely work with $\partial_{0,\nu}$ in the PDE so that instead of solving $(\partial_{0,\nu}M(\partial_{0,\nu}^{-1})+A)u=f$, we could also solve 
\begin{equation}\label{eq:intpde}
  (\partial_{0,\nu}M(\partial_{0,\nu}^{-1})+A)v=\partial_{0,\nu}^{-1}f,
\end{equation}
and obtain the original solution $u=\partial_{0,\nu} v$. This will be advantageous when dealing with irregular right-hand sides, especially stochastic ones. In particular, $\partial_{0,\nu}^{-1}$ (and scalar functions thereof) commute with the solution operator $S_\nu$ given in Theorem \ref{thm:solth}. Thus (see also \cite[Theorem 6.2.5]{picardbook}), the solution theory obtained in Theorem \ref{thm:solth} carries over to $H_{\nu,k}(\mathbb{R};H)$, that is, the solution operator $S_\nu$ admits a continuous linear extension to all $H_{\nu,k}$-spaces:
\[
   S_\nu \in L(H_{\nu,k}(\mathbb{R};H)) \quad(k\in \mathbb{Z}).
\]

(b) It can be shown that for all $\eps>0$ and $u\in \dom(\overline{\partial_{0,\nu} M(\partial_{0,\nu}^{-1})+A})$ we have that $(1+\eps\partial_{0,\nu})^{-1}u\in \dom(\partial_{0,\nu})\cap \dom(A)$, see \cite[Theorem 6.2.5]{picardbook} or \cite[Lemma 5.2]{Waurick2014MMAS_Non}.

(c) With the notion of Sobolev chains as introduced in the previous section, we may neglect the closure bar in 
\begin{equation}\label{eq:epo}
  \overline{(\partial_{0,\nu}M(\partial_{0,\nu}^{-1})+A)}u=f.
\end{equation}
Indeed, the latter equation holds in $H_{\nu,0}(\mathbb{R};H)$, but, since \[H_{\nu,0}(\mathbb{R};H))\hookrightarrow H_{\nu,-1}(\mathbb{R};H)\cap H_{\nu,0}(\mathbb{R};H_{-1}(A+1))\text{ continuously,}\] we obtain equality \eqref{eq:epo} also in the space $H_{\nu,-1}(\mathbb{R};H)\cap H_{\nu,0}(\mathbb{R};H_{-1}(A+1))$. Moreover, for $u\in H_{\nu,0}(\mathbb{R};H)$, we have $\partial_{0,\nu}M(\partial_{0,\nu}^{-1})u\in H_{\nu,-1}(\mathbb{R};H)$ and $Au\in H_{\nu,0}(\mathbb{R};H_{-1}(A+1))$. Thus, 
\[
   \overline{(\partial_{0,\nu}M(\partial_{0,\nu}^{-1})+A)}u = \partial_{0,\nu}M(\partial_{0,\nu}^{-1})u+Au.
\]
In fact, in the proof of Theorem \ref{thm:solth} we have shown that $\dom(\partial_{0,\nu})\cap \dom(A)$ is dense in \[\dom(\overline{(\partial_{0,\nu}M(\partial_{0,\nu}^{-1})+A)})\] with respect to the graph norm of $\overline{(\partial_{0,\nu}M(\partial_{0,\nu}^{-1})+A)}$. But, a sequence $(u_n)_n$ converging to $u$ in $\dom(\overline{(\partial_{0,\nu}M(\partial_{0,\nu}^{-1})+A)})$ converges to $u$ particularly in $H_{\nu,0}(\mathbb{R};H)$. So, since $\partial_{0,\nu}\colon H_{\nu,0}(\mathbb{R};H)\to H_{\nu,-1}(\mathbb{R};H)$ and $A\colon H_{\nu,0}(\mathbb{R};H)\to H_{\nu,0}(\mathbb{R};H_{-1}(A+1))$ are continuous, we obtain
\[
   f=\lim_{n\to \infty} (\partial_{0,\nu}M(\partial_{0,\nu}^{-1})+A)u_n = \lim_{n\to\infty}\partial_{0,\nu}M(\partial_{0,\nu}^{-1})u_n+\lim_{n\to\infty}Au_n=\partial_{0,\nu}M(\partial_{0,\nu}^{-1})u+Au
\]
with limits computed in $H_{\nu,-1}(\mathbb{R};H)\cap H_{\nu,0}(\mathbb{R};H_{-1}(A+1))$.
\end{rem}

We shall now sketch how to deal with initial value problems. In fact, until now we have only considered equations like 
\[ (\partial_{0,\nu}M(\partial_{0,\nu}^{-1})+A)u=f, \]
that is, equations with a source term on the right-hand side, and some boundary conditions encoded in the domain of the (partial differential) operator $A$, but no initial conditions. In fact, we will show now, how to incorporate them into the right-hand side of the PDE. For a simple case, we rephrase the arguments in \cite[Section 6.2.5, Theorem 6.2.9]{picardbook}.

Take $u_0\in \dom(A)$, and $f\in H_{\nu,0}(\R;H)$ with $f$ vanishing on $(-\infty,0]$. Then, our formulation for initial value problems is as follows. For the sake of presentation, we let $M(\partial_{0,\nu}^{-1})=M_0+\partial_{0,\nu}^{-1}M_1$ for some self-adjoint, non-negative $M_0\in L(H)$ and some $M_1\in L(H)$ satisfying $\nu M_0+ \Re M_1\geq c$. An example for this would be $M_0=1_H$ and $M_1=0$. Consider
\[
   (\partial_{0,\nu}M_0 + M_1 +A ) v = f- \chi_{[0,\infty)}M_1 u_0 - \chi_{[0,\infty)}Au_0.
\]
Note that due to the exponential weight, we have $\chi_{[0,\infty)}M_1 u_0 + \chi_{[0,\infty)}Au_0\in H_{\nu,0}(\R;H)$. The solution theory in Theorem \ref{thm:solth} gives us a unique solution $v\in H_{\nu,0}(\R;H)$. Moreover, $v$ is supported on $[0,\infty)$, due to causality. 

\begin{lem}\label{lem:IVP}
  With the notation above, $u\coloneqq v+\chi_{[0,\infty)}u_0$ solves the initial value problem 
\[
     \begin{cases}
           \overline{(\partial_{0,\nu}M_0 + M_1 +A )} u = f & \text{ on }(0,\infty)\\
           (M_0u)(0+)=M_0u_0.&\text{ in }H_{-1}(A+1)
     \end{cases}
\]
\end{lem}
\begin{proof}
Note that on $(0,\infty)$ we get
\begin{align*}
 f- \chi_{[0,\infty)}M_1 u_0 - \chi_{[0,\infty)}Au_0 & =(\partial_{0,\nu}M_0 + M_1 +A )(u-\chi_{[0,\infty)}u_0)
 \\ & = \partial_{0,\nu}M_0(u-\chi_{[0,\infty)}u_0) + (M_1 +A )(u-\chi_{[0,\infty)}u_0),
\end{align*}
where these equalities hold in $H_{\nu,-1}(\R;H)\cap H_{\nu,0}(\R;H_{-1}(A+1))$. Hence, as $\partial_{0,\nu}M_0\chi_{[0,\infty)}u_0$ vanishes on $(0,\infty)$, we arrive at
\[
  f = (\partial_{0,\nu}M_0+ M_1 +A )u\text{ on }(0,\infty).
\]
It remains to check whether the initial datum is attained. From the equation
\[(\partial_{0,\nu}M_0 + M_1 +A ) v = f- \chi_{[0,\infty)}M_1 u_0 - \chi_{[0,\infty)}Au_0 \]
we see that $\partial_{0,\nu}M_0v\in H_{\nu,0}( \R;H_{-1}(A+1))$. Thus, $M_0v\in H_{\nu,1}( \R;H_{-1}(A+1))$. By the Sobolev embedding theorem (see e.g.~\cite[Lemma 5.2]{Kalauch}), we infer $M_0v \in C(\R;H_{-1}(A+1))$. In particular, we get
\[
   (M_0v)(0-)=(M_0v)(0+)
\]
with limits in $H_{-1}(A+1)$. By causality, $M_0v(0-)=0$ and, thus, we arrive at
\[
   0=M_0(u-\chi_{[0,\infty)}u_0)(0+),
\]
which gives $(M_0u)(0+)=M_0u_0$, that is, the initial value is attained in $H_{-1}(A+1)$.
\end{proof}

The results above enable us to solve linear partial differential equations with initial conditions just by looking at non-homogeneous problems with $H_{\nu,0}$ right-hand sides. A few comments are in order.

\begin{rem}\label{rem:indep_of_solop}
  (a) The solution operator in Theorem \ref{thm:solth} is independent of $\nu$, in the following sense: let $\nu,\mu$ be sufficiently large and denote the corresponding solution operators by $S_\nu$ and $S_\mu$ respectively. Then for $f\in H_{\nu,0}(\R;H)\cap H_{\mu,0}(\R;H)$ we have $S_\nu f= S_\mu f$, see e.g.\ \cite[Theorem 6.1.4]{picardbook} or \cite[Lemma 3.6]{Trostorffexpstab} for a detailed proof. Therefore we shall occasionally drop the index $\nu$ in the time-derivative or the solution operator if there is no risk of confusion.
  
  (b) For the sake of presentation, we state the above treatment of the deterministic PDEs in a rather restricted way. In fact the solution theory mentioned in Theorem \ref{thm:solth} can be generalized to maximal monotone relations $A$, see e.g.\ \cite{Trostorff_nonl}, or to non-autonomous coefficients, see \cite{PTWW,Waurick2014MMAS_Non}. For our purposes of investigating random right-hand sides however, Theorem \ref{thm:solth} is sufficient.
\end{rem}

Now we present the last ingredient before turning to stochastic PDEs. We shall present a perturbation result which will help us to deduce well-posedness of stochastic partial differential equations, where we interpret the stochastic part as a nonlinear perturbation on the right-hand side of the PDE. In order to do so, we give a definition of so-called evolutionary mappings, which is a slight variant of the notions presented in \cite[Definition 2.1]{WaurickGcon} and \cite[Definition 4.7]{Kalauch}.

\begin{defn}\label{def:evolutionary} Let $H,G$ Hilbert spaces, $\nu_0>0$.
Let
\[ F: \dom(F)\subseteq \bigcap_{\nu\geq0} H_{\nu,0}(\R;H)\to \bigcap_{\nu\geq\nu_0}H_{\nu,0}(\R;G),\] 
where $\dom(F)$ is supposed to be a vector space. We call $F$ \emph{evolutionary (at $\nu_0$)}, if for all $\nu\geq\nu_0$, $F$ satisfies the following properties
\begin{enumerate}[label=(\roman{enumi}),ref=(\roman{enumi})]
 \item $F$ is Lipschitz-continuous as a mapping 
 \[ F_{0,\nu} \colon \dom(F)\subseteq H_{\nu,0}(\R;H)\to H_{\nu,0}(\R;G),\,\phi\mapsto F(\phi), \]
 \item $\|F\|_{\textrm{ev},\textrm{Lip}}\coloneqq \limsup_{\nu\to \infty} \|F_\nu\|_{\textrm{Lip}}<\infty$, with $F_\nu \coloneqq \overline{F_{0,\nu}}$.
\end{enumerate}
The non-negative number $\|F\|_{\textrm{ev},\textrm{Lip}}$ is called the \emph{the eventual Lipschitz constant of $F$}.

If, in addition, $F_\nu$ leaves $\dom(F_\nu)=\overline{\dom(F)}^{H_{\nu,0}}$ invariant, then we call $F$ \emph{invariant evolutionary (at $\nu_0$)}. 
\end{defn}
 
Similar to the solution operator $S_\nu$ to certain partial differential equations (see Remark \ref{rem:indep_of_solop}), evolutionary mappings are independent of $\nu$ in the following sense:
 
\begin{lem}\label{rem:indep_of_evo}
Let $F$ be evolutionary at $\nu_0>0$. Assume that multiplication by the cut-off function $\chi_{(-\infty,a]}$ leaves the space $\dom(F)$ invariant for all $a\in \R$, that is, for all $a\in \mathbb{R}$, $\phi\in \dom(F)$
\begin{equation}\label{eq:inv}
   \chi_{(-\infty,a]}\phi\subseteq \dom(F).
\end{equation}
Then $F_\nu|_{\dom(F_\nu)\cap \dom(F_\mu)}=F_\mu|_{\dom(F_\nu)\cap \dom(F_\mu)}$ for all $\nu\geq\mu\geq\nu_0$. 
\begin{proof}
Take $u\in \dom(F_\nu)\cap \dom(F_\mu)$ and assume as a first step, that $\chi_{(-\infty,a]}u=0$ for some $a\in \R$. By definition, there exists $(\phi_n)_n$ in $\dom(F)$ such that $\phi_n\to u$ in $H_{\mu,0}(\R;H)$. As $\dom(F)$ is a vector space and by being left invariant by multiplication by the cut-off function, we also have that $\psi_n\coloneqq\chi_{(a,\infty)}\phi_n\in \dom(F)$ as well as $\psi_n\to u$ in $H_{\mu,0}(\R;H)$. From $\nu>\mu$, we infer that $\psi_n\to u$ in $H_{\nu,0}(\R;H)$. Hence, as $H_{\nu,0}(\R;G)$ and $H_{\mu,0}(\R;G)$ are continuously embedded in $L^2_{\textrm{loc}}(\lambda_\R;G)$,
\[
   F_\mu(u)=\lim_{n\to\infty} F_\mu(\psi_n)=\lim_{n\to\infty} F_\nu(\psi_n)= F_\nu(u).
\]
For general $u\in \dom(F_\nu)\cap \dom(F_\mu)$, note that the sequence $(u_n)_{n\in\N}\coloneqq (\chi_{[-n,\infty)}u)_{n\in\N}$ converges in both spaces $H_{\nu,0}(\R;H)$ and $H_{\mu,0}(\R;H)$ by dominated convergence. The continuity of $F_\nu$ and $F_\mu$ implies convergence of $(F_\nu(u_n))_{n\in\N}$ and $(F_\mu(u_n))_{n\in\N}$ in $H_{\nu,0}(\R;G)$ and $H_{\mu,0}(\R;G)$, respectively. Therefore we get equality of the respective limits by $F_\nu(u_n)=F_\mu(u_n)$ by the arguments in the first step of this proof, again by the fact that both spaces $H_{\nu,0}(\R;G)$ and $H_{\mu,0}(\R;G)$ are continuously embedded in $L^2_{\textrm{loc}}(\lambda_\R;G)$.  
\end{proof}   
\end{lem}

\begin{rem}
	In both articles \cite[Definition 2.1]{WaurickGcon} and \cite[Definition 4.7]{Kalauch}, where the notion of evolutionary mappings was used, we assumed that the mappings under considerations are densely defined (and linear). Hence, the invariance condition is superfluous. But in the context of stochastic partial differential equations, one should think of $F$ to be a stochastic integral. This, however, is only a Lipschitz continuous mapping, if the processes to be integrated are adapted to the filtration given by the integrating process. The adapted processes form a closed subspace of all stochastic processes, and they will play the role of $\dom(F_\nu)$. This shall be specified in the next section. 
\end{rem}
 
As the final statement of this section, we provide the perturbation result which is applicable to stochastic partial differential equations.

\begin{cor}\label{cor:perres} Let $H$ be a Hilbert space, $\nu_0>0$. Assume that $F$ is invariant evolutionary (at $\nu_0$) as in Definition \ref{def:evolutionary} for $G=H$. Let $r>\frac{1}{2\nu_0}$, and suppose that $M\colon B(r,r)\to L(H)$ is analytic and bounded, satisfying 
\[
   \Re \langle z^{-1} M(z )\phi,\phi\rangle_H  \geq c \|\phi\|_H^2,
\]
for all $z\in B(r,r)$, all $\phi\in H$ and some $c>0$. Assume that $\|F\|_{\textnormal{ev},\textnormal{Lip}}< c$ and that $F_\nu$ is causal for all $\nu>\nu_0$. Furthermore suppose that for all $\nu>\nu_0$, we have $S_\nu\phi\subseteq \dom(F_\nu)$, for all $\phi\in \dom(F_\nu)$ with $S_\nu$ from Theorem \ref{thm:solth}.

Then the mapping
\begin{align*}
   \Phi_\nu \colon \dom(\Phi_\nu)\subseteq \dom(F_\nu) & \to \dom(F_\nu)\\
        u & \mapsto \overline{(\partial_{0,\nu}M(\partial_{0,\nu}^{-1})+A)}u+F_\nu(u)
\end{align*}
with domain 
\[
   \dom(\Phi_\nu) = \left\{ u\in \dom(F_\nu); \overline{(\partial_{0,\nu}M(\partial_{0,\nu}^{-1})+A)}u+F_\nu(u) \in \dom(F_\nu)\right\} 
\]
admits a Lipschitz-continuous inverse mapping defined on the whole of $\dom(F_\nu)$ for all $\nu>\nu_0$ large enough. Moreover, $\Phi_\nu^{-1}$ is causal.
\end{cor}
\begin{proof}
  Choose $\nu>\nu_0$ so large such that $\|F_\nu\|_{\textrm{Lip}}<c$ and let $f\in \dom(F_\nu)$. Now, $u\in \dom(\Phi_\nu)$ satisfies 
\[
 \overline{(\partial_{0,\nu}M(\partial_{0,\nu}^{-1})+A)}u+F_\nu(u) =f
\]
if and only if $u$ is a fixed point of the mapping
\[
  \Psi:\dom(F_\nu)\to \dom(F_\nu), x\mapsto \overline{(\partial_{0,\nu}M(\partial_{0,\nu}^{-1})+A)}^{-1}\left(f-F_\nu(x)\right).
\]
Note that, since $\dom(F_\nu)$ is a vector space, $\Psi$ is in fact well-defined. Moreover, $\Psi$ is a contraction, by the choice of $\nu$. Indeed, let $u,v\in \dom(F_\nu)$ then
\begin{align*}
   \|\Psi(u)-\Psi(v)\|& = \|\overline{(\partial_{0,\nu}M(\partial_{0,\nu}^{-1})+A)}^{-1}\left(f-F_\nu(u)\right)-\overline{(\partial_{0,\nu}M(\partial_{0,\nu}^{-1})+A)}^{-1}\left(f-F_\nu(v)\right)\|
   \\  & = \|\overline{(\partial_{0,\nu}M(\partial_{0,\nu}^{-1})+A)}^{-1}\left(F_\nu(u)\right)-\overline{(\partial_{0,\nu}M(\partial_{0,\nu}^{-1})+A)}^{-1}\left(F_\nu(v)\right)\|
   \\  & \leq \|\overline{(\partial_{0,\nu}M(\partial_{0,\nu}^{-1})+A)}^{-1}\|\|\left(F_\nu(u)\right)-\left(F_\nu(v)\right)\|
   \\  & \leq \frac{1}{c}\|\left(F_\nu(u)\right)-\left(F_\nu(v)\right)\|\leq \frac{1}{c}\|F_\nu\|_{\textrm{Lip}}\|u-v\|,
\end{align*}
so $\Psi$ is strictly contractive as $\|F_\nu\|_{\textrm{Lip}}<c$. Hence, the inverse of $\Phi_\nu$ is a well-defined Lipschitz continuous mapping, by the contraction mapping principle. 

Next, we show causality of the solution operator. For this it suffices to observe that $\Psi$ is causal. But, by Theorem \ref{thm:solth}, $\Psi$ is a composition of causal mappings, yielding the causality for $\Psi$ and, hence, the same for the solution mapping of the equation under consideration in the present corollary.
\end{proof}

As already mentioned, we use the above perturbation result to conclude well-posedness of stochastically perturbed partial differential equations. In the application, we have in mind, the invariance of $\dom(F_\nu)$ under $S_\nu$ is a consequence of causality. In fact, we will have that $\dom(F_\nu)$ is the restriction of $H_{\nu,0}$-functions to the class of predictable processes. A remark on the dependence of $\Phi^{-1}_\nu$ on $\nu$ is in order.

\begin{rem}\label{rem:indi}
In order to show independence of $\nu$, that is, $\Phi_\nu^{-1}f=\Phi_\mu^{-1}f$ for $f\in \dom(F_\nu)\cap \dom(F_\mu)$ for $\nu,\mu$ chosen large enough, we need to assume condition \eqref{eq:inv} in addition. Indeed, take $f\in \dom(F_\nu)\cap \dom(F_\mu)$ for $\nu,\mu$ sufficiently large as in Corollary \ref{cor:perres}.  Then, with $\Psi_\nu \coloneqq \overline{(\partial_{0,\nu}M(\partial_{0,\nu}^{-1})+A)}^{-1}\left(f-F_\nu(\cdot)\right)$ (and similarly for $\Psi_\mu$), the proof of Corollary \ref{cor:perres} shows that $\Phi_\nu^{-1}(f)=\lim_{n\to\infty} \Psi_\nu^n(f)$.  But, as $f\in \dom(F_\nu)\cap \dom(F_\mu)$, we get 
\[
   \Psi_\nu (f)= S_\nu (f-F_\nu(f))=S_\nu(f-F_\mu(f))=S_\mu(f-F_\mu(f))=\Psi_\mu(f),
\]
by Remark \ref{rem:indep_of_solop}(a) and Lemma \ref{rem:indep_of_evo}. In particular, $\Psi_\nu (f)\in \dom(F_\nu)\cap \dom(F_\mu)$. In the same way, one infers that $\Psi_\nu^n(f)=\Psi_\mu^n(f)$ for all $n\in \N$. Consequently, $\Phi_\nu^{-1}(f)$ and $\Phi_\mu^{-1}(f)$ are limits of the same sequence in $\dom(F_\nu)$ and $\dom(F_\mu)$, respectively. Thus, as both the latter spaces are continuously embedded into $L^2_{\textrm{loc}}(\lambda_\R;H)$ these limits coincide.
\end{rem}

Due to Remark \ref{rem:indi}, in what follows, we will not keep track on the value of $\nu>0$ in the notation of the operators involved as it will be clear from the context in which Hilbert space the operators are established in.

\section{Application to SPDEs}\label{sec:stochastics}
In this section we show how to apply the solution theory from Section \ref{sec:soltheory} to an SPDE of the form \eqref{eq:SPDE}. The basic idea is to replace the Hilbert space $H$ in Section \ref{sec:soltheory} by $L^2(\P)\otimes H(\cong L^2(\P;H))$, where $L^2(\P) = L^2(\Omega,\scrA,\P)$ is the $L^2$-space of a probability space $(\Omega,\scrA,\P)$ and $H$ is the Hilbert space where the (unbounded) operator $A$ is thought of as being initially defined. A typical choice would be $H=L^2(\lambda_D )^{d+1}$, for some open $D\subseteq\Rd$, and $A$ being some differential operator, but also more general operator equations are possible.

As already mentioned in the introduction, we consider the stochastic integral on the right-hand side as a perturbation of the deterministic partial differential equation. Therefore we need to make sense of the term 
\[ \big(\partial_0M(\partial_0^{-1})+A\big)^{-1}\bigg(\int_0^t \sigma(u(s))dW(s)\bigg). \]
In Section \ref{sec:SI} we establish that this term is well-defined, and after that, in Section \ref{sec:multiplicative}, we can use this to treat SPDEs with multiplicative noise using the fixed-point argument carried out in Corollary \ref{cor:perres}. In principle, using this idea, one can also treat SPDEs with additive noise, but a slightly different analysis in Section \ref{sec:additive} also gives us a result on SPDEs with additive noise.

\subsection{Treatment of the stochastic integral}\label{sec:SI}
The concept of stochastic integration we use in the following is the same as in \cite{dapratozabczyk}, and we repeat the most important points here.

\begin{defn}[Wiener process]\label{defn:Bm}
  Let $G$ be a separable Hilbert space, $(e_k)_{k\in\N}$ an orthonormal basis of $G$, $(\lambda_k)_{k\in\N}\in \ell_1(\N)$ with $\lambda_k\geq0$ for all $k\in \N$, and let $(W_k)_{k\in \N}$ be a sequence of independent real-valued Brownian motions. For $t\in [0,\infty)$ we define the \emph{$G$-valued Wiener process}, by
  \[
     W(t) =\sum_{k=1}^\infty \sqrt{\lambda_k} W_k(t)e_k,
  \]
  and we set $W(t)=0$ for all negative times $t\in (-\infty,0)$.
\end{defn}

In order to make sense of stochastic integration we reinterpret the notion of a filtration in an operator-theoretic way. We will use the notation $A\leq B$ for two bounded linear operators on a Hilbert space $H$ if $\langle Ax,x\rangle\leq \langle Bx,x\rangle$ for all $x\in H$. 

\begin{defn}[Filtration and predictable processes] (a) Let $H$ be a Hilbert space, $P=(P_t)_{t\geq 0}$ is called a \emph{filtration on $H$}, if for all $t\in \R_+$ the operator $P_t$ is an orthogonal projection, $P_s\leq P_t\leq 1_H$ for all $s\leq t$.

(b) Let $\nu>0$, $Z\colon\R\to H$. We call $Z$ \emph{predictable (with respect to $P$)}, if 
\[
   Z\in H_{\nu,0}(\R;P)\coloneqq \overline{S_P},
\]
where
\begin{equation}\label{eq:simple_proc}
 S_P\coloneqq \lin\{ \chi_{(s,t]}\phi;\phi\in \ran(P_s),s,t\in \R, s<t\}
\end{equation}
and the closure is taken in $H_{\nu,0}(\R;H)$.

(c) Let $G$ be a Hilbert space. We say that $Z\colon \R \to H\otimes G$ is \emph{predictable (with respect to $P$) with values in $G$}, if $Z$ is predictable with respect to $P\otimes 1_G\coloneqq (P_t\otimes 1_G)_{t\in \R}$.
\end{defn}

\begin{rem}\label{rem:filtisfilter} In applications, $H=L^2(\P)$ for some probability space $(\Omega,\scrA,\P)$ and $(P_t)_{t\in\R}$ is given by a family of nested $\sigma$-algebras $(\mathcal{F}_t)_{t\in \R}$. More precisely, 
\begin{equation}\label{eq:depexp}
   P_t \colon L^2(\P) \to L^2(\P), X\mapsto \E(X|\mathcal{F}_t)\quad (t\in \mathbb{R}).
\end{equation}
In particular, if we are given a  $G$-valued Wiener process $W$ with underlying probability space $(\Omega,\scrA,\P)$ as in Definition \ref{defn:Bm}, the natural filtration is given by $\mathcal{F}_t \coloneqq \sigma(W_k(s);k\in\N,-\infty < s\leq t)$, $t\in \R$. The corresponding family of projections $P_W=(P_t)_t$ is then given as in \eqref{eq:depexp}. Hence, $S_{P_W}$ (see also \eqref{eq:simple_proc}) reads
\begin{align*}
   S_{P_W} & = \lin\{ \chi_{(s,t]}\phi;\phi\in \ran(P_s),s,t\in \R, s<t\}
   \\ &      = \lin\{ \chi_{(s,t]}\phi;P_s\phi = \phi,s,t\in \R, s<t\}
   \\ &      = \lin\{ \chi_{(s,t]}\phi;\E(\phi|\mathcal{F}_s) = \phi,s,t\in \R, s<t\}
   \\ &      = \lin\{ \chi_{(s,t]}\phi;\phi\text{ is }\mathcal{F}_s\text{-measurable, }s,t\in \R, s<t\}.
\end{align*}
Note that $S_{P_W}$ are also called simple \emph{predictable processes}. In this case, one could also take $\scrA = \scrF_\infty \coloneqq \sigma\left(\bigcup_{t\geq0} \scrF_t\right)$.
\end{rem}

For later use, we also have to show that the solution map as defined in Section \ref{sec:soltheory} does not destroy the predictability. This is however a direct consequence of the causality of the solution map stated in Theorem \ref{thm:solth}. 
For ease of presentation, we will freely identify $H\otimes H_1$ with $H_1\otimes H$ and $H_{\nu,0}(\R;H)$ with $H_{\nu,0}(\R)\otimes H$. In particular, this effects the following loose notation: a continuous operator $M$ on $H$ is then extended to a continuous linear operator on $H_1\otimes H$ by $M\otimes 1_{H_1}$ (and of course by $1_{H_1}\otimes M$), if we want to stress that it is extended at all (and not simply write $M$).

\begin{thm}\label{thm:causal,adap} Let $H,G$ be Hilbert spaces, $P$ a filtration on $H$. Let $M\colon H_{\nu,0}(\R;G)\to H_{\nu,0}(\R;G)$ be a causal, continuous linear operator. Then the canonical extension of $M$ to $H_{\nu,0}(\R;G)\otimes H$ leaves the space of predictable processes invariant, that is,
\[
  M\otimes 1_H \left[H_{\nu,0}(\R;P\otimes 1_G)\right]\subseteq H_{\nu,0}(\R;P\otimes 1_G).
\]
\end{thm}
\begin{proof}
 By continuity of $M$, it suffices to prove $(M\otimes 1_H)[S_{P\otimes 1_{G}}]\subseteq H_{\nu,0}(\R;P\otimes 1_G)$. Let $f\in S_{P\otimes 1_{G}}$. By linearity of $M\otimes 1_H$, we may assume without loss of generality that $f=\chi_{(s,t]}\eta$ for some $\eta=(P_s\otimes 1_G)(\eta)\in H\otimes G$ and $s,t\in \mathbb{R}$. Next, by the density of the algebraic tensor product of $H$ and $G$, we find sequences $(\phi_n)_n$ in $H$ and $(\psi_n)_n$ in $G$ with the property
 \[
    \sum_{n=1}^\infty \phi_n\otimes \psi_n =\eta\in H \otimes G.
 \]
 But, 
 \[
    \eta = (P_s\otimes 1_G)(\eta) = (P_s\otimes 1_G) \sum_{n=1}^\infty \phi_n\otimes \psi_n = \sum_{n=1}^\infty (P_s\phi_n)\otimes \psi_n.
 \]
 Thus, without restriction, we may assume that $\phi_n=P_s\phi_n$ for all $n\in\mathbb{N}$. Since, by definition, the predictable mappings form a closed subset of $H_{\nu,0}(\R;H\otimes G)$ and $M\otimes 1_H$ is continuous, it suffices to prove that for all $N\in \mathbb{N}$, 
 \[
    (M\otimes 1_H) \left(\chi_{(s,t]}\sum_{n=1}^N \phi_n\otimes \psi_n\right)
 \]
 is predictable. By linearity of $M\otimes 1_H$, we are left with showing that $(M\otimes 1_H) \left(\chi_{(s,t]}\phi\otimes \psi\right)$ is predictable for all $\phi\in \ran(P_s)$ and $\psi\in G$. Note that
 \[
    (M\otimes 1_H) \left(\chi_{(s,t]}\phi\otimes \psi\right) = (M(\chi_{(s,t]}\psi))\otimes \phi.
 \]
Next, causality of $M$ implies that $\spt M(\chi_{(s,t]}\psi) \subseteq [s,\infty)$, where $\spt$ denotes the support of a function. We conclude with the observation that $M(\chi_{(s,t]}\psi)$ can be approximated by simple functions in $H_{\nu,0}(\R;G)$ supported on $(s,\infty)$ only. Hence, $(M(\chi_{(s,t]}\psi))\otimes \phi$ is predictable.\end{proof}

Next, we come to the discussion of the stochastic integral involved:

\begin{defn}[stochastic integral]\label{defn:stoch_int}
  Let $H$, $G$ be separable Hilbert spaces\[W=\sum_{k=1}^\infty \sqrt{\lambda_k} W_k(\cdot)e_k\] a $G$-valued Wiener process. Let $Z$ be a predictable stochastic process with respect to the natural filtration induced by $W$ as in Remark \ref{rem:filtisfilter} with values in $L_2(G,H)$, the space of Hilbert--Schmidt operators from $G$ to $H$. Then we define \emph{the stochastic integral of $Z$ with respect to $W$} for all $t\in[0,\infty)$ as follows
  \[
    \int_0^t Z(s) dW(s)\coloneqq \sum_{k\in\N} \lambda_k^{1/2}\int_0^t Z(s)e_k dW_k(s).
  \] 
  We put $\int_0^t Z(s) dW(s)\coloneqq 0$ for all $t<0$.
\end{defn}

\begin{rem}[It\^o isometry]
  In the situation of Definition \ref{defn:stoch_int}, the following It\^o isometry holds
  \[ \E\bigg[\bigg\|\int_0^t Z(s)dW(s)\bigg\|_H^2\bigg] = \E\bigg[\int_0^t \|Z(s)\|^2_{L_2(G,H)}ds\bigg]. \]
  Moreover, the stochastic integral seen as a process in $t\in\R$ is continuous and predictable with values in $H$, see \cite[Chapter 4]{dapratozabczyk} for details.
\end{rem}

Next, we will show the assumptions in Corollary \ref{cor:perres} applied to $F\colon u\mapsto \int_0^{(\cdot)}\sigma(u)dW$ with suitable Lipschitz continuous $\sigma$. For this we need the following key observation; we recall also Remark \ref{rem:filtisfilter}. We denote the Hilbert--Schmidt norm also by $\|\cdot\|_{L_2}$.

\begin{thm}\label{thm:stoch_int_evo} 
	Let $G$ be a separable Hilbert space, and $W$ a $G$-valued Wiener process, $(\Omega,\scrA,\P)$ as its underlying probability space and $(\mathcal{F}_t)_t$ the natural filtration induced by $W$ and corresponding filtration $P_W=(\E(\cdot|\mathcal{F}_t))_t$ on $L^2(\P)$. Then the mapping
	\[ F\colon S_{P_W\otimes 1_{L_2(G,H)}} \subseteq H_{\nu,0}(\R;P_{W}\otimes 1_{L_2(G,H)})\to H_{\nu,0}(\R;P_W\otimes 1_H), \] 
	where $F$ is given by 
  \[
    F(Z)=\left(t\mapsto \int_0^t Z(s) dW(s)\right)
   \]
  is evolutionary at $\nu$ for all $\nu>0$ with eventual Lipschitz constant $0$.
\end{thm}
\begin{proof}
  Since $F$ is linear and maps simple predictable processes to predictable processes, it suffices to prove boundedness of $F$. In order to do so, let $Z\in S_{P_W\otimes 1_{L_2(G,H)}}$. Then, we get using Fubini's Theorem and the It\^o isometry,
\begin{align}\label{eq:LCSI}
	& \E\bigg[\bigg\|\int_0^\cdot Z(s) dW(s)\bigg\|^2_{\nu,0}\bigg] \notag\\
	& = \int_{\R} \E\bigg[\bigg\|\int_0^t Z(s) dW(s)\bigg\|^2_{H}\bigg]\exp(-2\nu t)d t \notag\\
	& = \int_{\R} \E\bigg[\int_0^t\big\|Z(s)\big\|_{L_2}^2 ds\bigg]\exp(-2\nu t)d t \notag\\
	& = \E\bigg[\int_\R\big\|Z(s)\big\|_{L_2}^2 \int_s^\infty\exp(-2\nu t)d td s\bigg] \notag\\
	& = \frac{1}{2\nu} \E\bigg[\int_\R\big\|Z(s)\big\|_{L_2}^2 \exp(-2\nu s) ds\bigg] \notag\\	
	& = \frac{1}{2\nu} \E\big[\|Z\|^2_{\nu,0}\big].
\end{align}
From this we see that $F$ is Lipschitz continuous, and that its Lipschitz constant goes to zero as $\nu\to\infty$.
\end{proof}

\begin{rem}[on space-time white noise]
  Note that in the proof of the previous theorem, the crucial ingredients are Fubini's Theorem and the It\^o isometry. The It\^o isometry is true also for the stochastic integral with the Wiener process attaining values in a possibly larger Hilbert space $G'\supseteq G$. Hence, the latter theorem remains true, if we consider space-time white noise instead of the white noise discussed in this exposition, see (in particular) \cite[formula (3.16)]{H09}.
\end{rem}

\subsection{SPDEs with multiplicative noise}\label{sec:multiplicative}
In this section we apply the solution theory presented in Section \ref{sec:soltheory} to equations with a stochastic integral. As already mentioned in the introduction, we consider equations of the form \eqref{eq:SPDE} with a stochastic integral instead of the more common random noise term $\sigma(u(t))\dot W(t)$. This formulation is however in line with the usual way of formulating an SPDE, since in some sense we consider ``a once integrated SPDE'' and we interpret the noise term $\partial_0^{-1}(\sigma(u(t))\dot W(t))$ as the stochastic integral in Hilbert spaces with respect to a cylindrical Wiener process denoted by 
\[ \partial_0^{-1}\big(\sigma(u(t))\dot W(t)\big) \coloneqq \int_0^t \sigma(u(s))\dot W(s)ds\coloneqq \int_0^t \sigma(u(s)) dW(s). \]
As a matter of convenience, we treat the case of zero initial conditions first. In Remark \ref{rem:different_Lipsch}(b) we shall comment on how non-vanishing initial data can be incorporated into our formulation.

\begin{thm}[Solution theory for (abstract) stochastic differential equations]\label{thm:spdes}
  Let $H$, $G$ be separable Hilbert spaces, and let $W$ be a $G$-valued Wiener process with underlying probability space $(\Omega,\scrA,\P)$. Assume that the filtration $P_W=(P_t)_{t}$ on $L^2(\P)$ is generated by $W$ (see Remark \ref{rem:filtisfilter}). Let $r>0$, and assume that $M\colon B(r,r)\to L(H)$ is an analytic and bounded function, satisfying 
  \begin{equation}\label{eq:condition_M}
    \Re\langle (z^{-1}M(z))\phi,\phi\rangle_H  \geq c\|\phi\|_H^2,
  \end{equation}
  for all $z\in B(r,r)$, $\phi\in H$ and some $c>0$. Let $A\colon\dom(A)\subseteq H\to H$ be skew-self-adjoint, and $\sigma\colon H\to L_2(G,H)$ with
  \[
     \|\sigma(u)-\sigma(v)\|_{L_2}\leq L\|u-v\|_H\quad(u,v\in H)
  \]
  for some $L\geq0$.

  Then there exists $\nu_1\geq0$ such that for all $\nu>\nu_1$, and $f\in H_{\nu,0}(\R;P_W\otimes 1_H)$ the equation
  \begin{equation}\label{eq:thm_spde}
    \overline{(\partial_{0,\nu}M(\partial_{0,\nu}^{-1})+A)}u= f+ \int_0^{\cdot}\sigma(u(s))dW(s)
  \end{equation}
  admits a unique solution $u\in H_{\nu,0}(\R;P_W\otimes 1_H)$. The solution does not depend on $\nu$ in the sense of Remark \ref{rem:indep_of_solop}.
  \end{thm}
\begin{proof}
We apply Corollary \ref{cor:perres} for $F\colon Z\mapsto \int_0^{\cdot}\sigma(Z)dW(s)$. By Theorem \ref{thm:stoch_int_evo} and the Lipschitz continuity of $\sigma$, we infer that $F$ is invariant evolutionary with eventual Lipschitz constant being $0$. Indeed, since $W(t)=0$ for $t<0$, we may write
\[
   \int_0^{\cdot}\sigma(u(s))dW(s) = \int_0^{\cdot}\chi_{[0,\infty)}(s)\sigma(u(s))dW(s).
\]
But, for all $\nu>0$, $s\mapsto \chi_{[0,\infty)}(s)\sigma(0) \in H_{\nu,0}(\R;P_W\otimes 1_{L_2(G;H)})$ and, therefore, we get for $u\in H_{\nu,0}(\R;P_W\otimes 1_{H})$
\begin{align*}
   & \|\chi_{[0,\infty)}(\cdot)\sigma(u(\cdot))\|_{H_{\nu,0}(\R;P_W\otimes 1_{L_2(G;H)})}
   \\& \leq \|\chi_{[0,\infty)}(\cdot)\sigma(u(\cdot))-\chi_{[0,\infty)}(\cdot)\sigma(0)\|_{H_{\nu,0}(\R;P_W\otimes 1_{L_2(G;H)})}+\|\chi_{[0,\infty)}(\cdot)\sigma(0)\|_{H_{\nu,0}(\R;P_W\otimes 1_{L_2(G;H)})} \\
    & \leq L\|\chi_{[0,\infty)}(\cdot)u(\cdot)\|_{H_{\nu,0}(\R;P_W\otimes 1_{H})}+\|\chi_{[0,\infty)}(\cdot)\sigma(0)\|_{H_{\nu,0}(\R;P_W\otimes 1_{L_2(G;H)})}<\infty.
\end{align*}
Moreover, it is equally easy to see that 
\[
   \chi_{[0,\infty)}\sigma \colon H_{\nu,0}(\R;P_W\otimes 1_{H}) \to  H_{\nu,0}(\R;P_W\otimes 1_{L_2(G;H)}),  u\mapsto \chi_{[0,\infty)}(\cdot)\sigma(u(\cdot))
\]
is Lipschitz continuous with Lipschitz constant bounded by $L$. Hence, by Theorem \ref{thm:stoch_int_evo}, we obtain that
\[
    H_{\nu,0}(\R;P_W\otimes 1_{H}) \ni u \mapsto \int_0^{\cdot} \sigma(u(s))dW(s)\in H_{\nu,0}(\R;P_W\otimes 1_{H})
\]
is Lipschitz continuous with eventual Lipschitz constant $0$. By Theorem \ref{thm:causal,adap}, we obtain that 
\[
   S_\nu [\dom(F_\nu)]=S_\nu[H_{\nu,0}(\R;P_W\otimes 1_{H})]\subseteq H_{\nu,0}(\R;P_W\otimes 1_{H})
\]
with $S_\nu$ from Theorem \ref{thm:solth}. Hence, the assertion follows from Corollary \ref{cor:perres}. (The independence of the solution of the parameter $\nu$ follows from Lemma \ref{rem:indep_of_evo} because the  multiplication with a cut-off function leaves the space of predictable processes invariant.)
\end{proof}

\begin{rem}\label{rem:different_Lipsch}(a) The above result is of course stable under Lipschitz continuous perturbations of the right-hand side. Indeed, let $B$ be an invariant evolutionary mapping, leaving the space of predictable processes invariant, with $B$ being causal and with the property that the eventual Lipschitz constant of $u\mapsto B(u)+\int_0^{(\cdot)}\sigma(u)dW(s)$ is strictly less than $c>0$, then the assertion of Theorem \ref{thm:spdes} remains the same, if one considers the equation
  \begin{equation}\label{eq:rem_spde}
    \overline{(\partial_{0,\nu}M(\partial_{0,\nu}^{-1})+A)}u= f+ \int_0^{(\cdot)}\sigma(u(s))dW(s)+B(u)
  \end{equation}
instead of \eqref{eq:thm_spde}. 
 
 (b) (Initial value problems) Similarly to the deterministic case treated in Lemma \ref{lem:IVP}, we can also formulate initial value problems for the special case $M(\partial_{0,\nu}^{-1})=M_0 +\partial_{0,\nu}^{-1}M_1$. Indeed the following initial value problem 
 \[
     \begin{cases}
         \overline{(\partial_{0,\nu}M_0 + M_1 +A)}u=f+ \int_0^{\cdot}\sigma(u(s))dW(s),& \text{ on }(0,\infty)\\
         M_0 u (0+)=M_0u_0,&\text{ in  }H_{-1}(A+1)
     \end{cases}
 \]
with given adapted $H$-valued process $f$ vanishing on $(-\infty,0]$ can be rephrased identifying $M_0u_0\in H\otimes L^2(\P)$. With this notation, the initial value problem above can be reformulated as 
 \[
   (\partial_{0,\nu}M_0 +M_1 +A)v = f+\int_0^{\cdot}\sigma(v(s)+\chi_{[0,\infty)}(s)u_0)dW(s)-\chi_{[0,\infty)}M_1u_0-\chi_{[0,\infty)}Au_0
 \]
as our appropriate realization of the initial value problem. Note that the map 
\[ v\mapsto \int_0^{\cdot}\sigma(v(s)+\chi_{[0,\infty)}(s)u_0)dW(s) \]
is still invariant evolutionary. Thus, solving for $v\in H_{\nu,0}(\R;H\otimes L^2(\P))$ gives, follow the lines of Lemma \ref{lem:IVP}, that $M_0v(0-)=0=M_0v(0+)\in H_{-1}(A+1)\otimes L^2(\P)$, which eventually leads to the attainment of the initial value $M_0 u (0+)=M_0u_0$ in $H_{-1}(A+1)\otimes L^2(\P)$.
\end{rem}

\begin{ex}
  As a particular example for Remark \ref{rem:different_Lipsch}(a), any deterministic Lipschitz continuous mapping from $H$ with values in $H$, is an eligible right-hand side in \eqref{eq:rem_spde}. These mappings have been used in \cite[Chapter 7]{dapratozabczyk}. 
  \end{ex}

\subsection{SPDEs with additive noise}\label{sec:additive}
In this section we investigate the solution theory of equations with additive noise, that is, the stochastic integral on the right-hand side in \eqref{eq:SPDE} is replaced by a stochastic process $X$: Let, in this section, $X$ be any $H$-valued stochastic process, more specifically, the map $(t,\omega)\mapsto X(t,\omega)$ belongs to $H_{\nu,0}(\R;H\otimes L^2(\P))$. This includes in particular stochastic processes on Hilbert spaces that have continuous or c\`adl\`ag paths, and in particular L\'evy processes and fractional Brownian motions. Hence, the equation to be solved is given by 
  \[
    \overline{(\partial_{0,\nu}M(\partial_{0,\nu}^{-1})+A)}u= f+X
  \]
Then we can apply Theorem \ref{thm:solth} to these equations and we will obtain a unique solution -- for any stochastic process $X$ whose paths are in $H_{\nu,0}(\R;H\otimes L^2(\P))$, which is only a condition on the integrability of its paths.

Now we we are going to show a more general result. With the notation as in Theorem \ref{thm:solth}, we consider the equation
  \begin{equation}\label{eq:thm_spde_add_k}
    \overline{(\partial_{0,\nu}M(\partial_{0,\nu}^{-1})+A)}u= f+\partial_{0,\nu}^k X,
  \end{equation}
where the right-hand side is an element of $H_{\nu,-k}(\R;H\otimes L^2(\P))$, for all $k\in\N_0$. Then, the noise term is interpreted as the $k$-times distributional time derivative of the paths of the stochastic process $X$. The space $H_{\nu,-k}(\R;H\otimes L^2(\P))$ is the distribution space belonging to $\partial_{0,\nu}$ realized as an operator in $H_{\nu,0}(\R;H\otimes L^2(\P))$. The solution theory for such a class of equations is then a corollary to the general solution theory in Theorem \ref{thm:solth}.

\begin{thm}\label{thm:solth_add}
  Assume that $M$ and $A$ satisfy the conditions in Theorem \ref{thm:solth}. Suppose that $X$ is a $H$-valued stochastic process whose paths belong to $H_{\nu,0}(\R;H\otimes L^2(\P))$. Then there exists a unique solution $u$ to \eqref{eq:thm_spde_add_k} in $H_{\nu,-k}(\R;H\otimes L^2(\P))$.
 \end{thm}
\begin{proof}  
The assertion follows once observed that $(\partial_0M(\partial_0^{-1})+A)^{-1}$ can be realized as a continuous linear operator in $H_{\nu,-k}(\R;H\otimes L^2(\P))$ with Lipschitz constant bounded above by $1/c$ (see also Remark \ref{rem:integratedPDE}).
\end{proof} 

We note the main achievement of this section. The right-hand side has to be in $H_{\nu,-k}(\R;H\otimes L^2(\P))$, only. Note that there are no stochastic integrals involved, neither did we make any assumption on the regularity of the noise term $\partial_{0,\nu}^k X$, other than that it is the $k$-th time-derivative of a stochastic process $X$. Therefore we have found a way to make sense of stochastic differential equations in Hilbert spaces where the random noise can be a very irregular object, given by the distributional derivative of any stochastic process (L\'evy, Markov etc.) with only the assumption of integrability of its paths. The solution to these equations is an element of the space of stochastic distributions (in the time argument).

\section{Examples}\label{sec:examples}
In this section, we shall give some examples for the solution theory presented above. We emphasize, that -- at least in principle -- the only thing to be taken care of is the formulation of the respective problem in an appropriate way as an operator equation in appropriate Hilbert spaces. The way how we do it is to start with the equation given formally as a stochastic differential equation and, after some algebraic manipulations, we shall give the appropriate replacement to be solved with the solution theory based on Theorem \ref{thm:spdes} or Theorem \ref{thm:solth_add}. In the whole section, we let $W$ be a $G$-valued Wiener process for some separable Hilbert space $G$ and we assume that $\sigma\colon H_0 \to L_2(G,H_0)$ is Lipschitz continuous, where $H_0$ will be clear from the context. For simplicity of the exposition, we assume that we only have a stochastic term containing $\sigma$ on the right-hand side and null initial conditions. The way how to incorporate a path-wise perturbation and/or non-zero initial conditions was shown in Remark \ref{rem:different_Lipsch}.

For the stochastic heat as well as for the stochastic wave equation, we justify our findings and put them into perspective of more classical solution concepts. For this, we note a general observation: Although the solutions constructed in this exposition live on the whole real time line, the support of the solutions is concentrated on the positive real axis provided the one of the right-hand side is. Indeed, this is a consequence of causality of the respective solution operators.

\subsection{Stochastic heat equation}\label{sec:SHE}
We consider the following SPDE in an open set $D\subseteq \Rd$
\begin{align}\label{eq:stoch_heat}
  \partial_0 u(t) - \Delta u(t) & = \sigma(u(t))\dot W(t), \notag\\
     u(0)=0, u|_{\partial D} & =0,
\end{align}
where $\Delta$ is the Laplace operator acting on the deterministic spatial variables $x\in D$ only. This equation has been studied in \cite{walsh}, see also \cite[Example 7.6]{dapratozabczyk} for a treatment in Hilbert spaces.  We establish the boundary condition in the way that $u\in H_0^1(D)$, the Sobolev space of the once weakly differentiable functions, which may be approximated in the $H^1(D)$-norm by smooth functions with compact support contained in $D$. Before we formulate the heat equation in our operator-theoretic setting, we need to introduce some differential operators.

\begin{defn}\label{def:operators}
   We define
   \begin{align*}
       \grad_c \colon C_c^\infty(D)\subseteq L^2(\lambda_D )&\to L^2(\lambda_D )^d\\
                                                   \phi&\mapsto (\partial_j\phi)_{j\in\{1,\ldots,d\}},\\
       \grad \colon H^1(D)\subseteq L^2(\lambda_D )&\to L^2(\lambda_D )^d\\
                                                   \phi&\mapsto (\partial_j\phi)_{j\in\{1,\ldots,d\}}                                                   
   \end{align*}
   and let $\div\coloneqq -\grad_c^*$, $\interior\div\coloneqq -\grad^*$ as well as $\interior\grad\coloneqq \overline{\grad}_c$.
\end{defn}

Throughout this section, we will use these operators to reformulate the SPDEs in an adequate way. The meaning of these operators is that the ones with the superscript ``$\interior{\ }$'' carry the homogeneous boundary conditions on $\partial D$: $\interior{\div}$ carries zero Neumann boundary conditions and $\interior{\grad}$ carries zero Dirichlet boundary conditions. With these operators we can rewrite the Laplacian with homogeneous Dirichlet boundary conditions as $\Delta = \div\interior{\grad}$.

We may now come back to the stochastic heat equation. We perform an algebraic manipulation to reformulate it as a system of first order SPDEs. First, we apply the operator $\partial_0^{-1}$ to equation \eqref{eq:stoch_heat}, see also Remark \ref{rem:integratedPDE}(a), and we arrive at 
\begin{equation}\label{eq:SHE1}
  u(t) - \partial_0^{-1}\Delta u(t) = \partial_0^{-1}\sigma(u(t))\dot W(t).
\end{equation}
We interpret the right-hand side as the following stochastic integral 
\[ \partial_0^{-1}\sigma(u(t))\dot W(t) := \int_0^\cdot \sigma(u)dW. \]
Observe that $\partial_0^{-1}$ and any spatial (partial differential) operator commute (see also Lemma \ref{l:dcp} below for a more precise statement). Therefore, formally, we can rewrite the second term in \eqref{eq:SHE1} as 
\[ - \partial_0^{-1}\Delta u = -\partial_0^{-1}\div\interior{\grad}u = -\div\partial_0^{-1}\interior{\grad}u. \]
Then, setting $q \coloneqq -\partial_0^{-1} \interior{\grad} u$, we arrive at the following first-order system 
\begin{equation}\label{eq:SHEsystem}
 \left(\partial_0 \begin{pmatrix}
   0 & 0 \\ 0 & 1
 \end{pmatrix} + \begin{pmatrix}
   1 & 0 \\ 0 & 0
 \end{pmatrix}+\begin{pmatrix}
   0 & \div \\ \interior{\grad} & 0
 \end{pmatrix}\right)\begin{pmatrix}
   u \\ q
 \end{pmatrix}=\begin{pmatrix}
   \int_0^{\cdot}\sigma(u)dW \\ 0
 \end{pmatrix},
\end{equation}
which we think of being an appropriate replacement for \eqref{eq:stoch_heat}. 
 
Assuming that $\sigma\colon L^2(\lambda_D )\to L_2(G,L^2(\lambda_D ))$ to be Lipschitz continuous, we can use Theorem \ref{thm:spdes} to show the existence and uniqueness of solutions to this system. The only things still to be checked are whether
 \[
  A = \begin{pmatrix}
   0 & \div \\ \interior{\grad} & 0
 \end{pmatrix}
 \]
 is skew-self-adjoint and whether 
 \[
   M(z)\coloneqq \begin{pmatrix}
   0 & 0 \\ 0 & 1
 \end{pmatrix} + z\begin{pmatrix}
   1 & 0 \\ 0 & 0
 \end{pmatrix} 
 \]
satisfies condition \eqref{eq:condition_M}, for some $r>0$. The former statement being easy to check using the definition of $\div$ and $\interior{\grad}$ as skew-adjoints of one another in Definition \ref{def:operators} and upon relying on Remark \ref{rem:canext}. In order to prove the validity of condition \eqref{eq:condition_M}, we let $(\phi,\psi)\in L^2(\lambda_D )\oplus L^2(\lambda_D )^d$ and compute using $z^{-1}=\ii t+\mu$ if $z\in B(r,r)$ for some $\mu>\frac{1}{2r}$ and $t\in \R$
\begin{multline}\label{eq:M(z)}
  \Re\Big\langle z^{-1}M(z)\begin{pmatrix}
                        \phi \\ \psi
                       \end{pmatrix}, \begin{pmatrix}
                        \phi \\ \psi
                       \end{pmatrix}\Big\rangle_{L^2(\lambda_D )^{d+1}} = \Re(\langle (it+\mu)\psi,\psi\rangle_{L^2(\lambda_D )^{d}})+\Re \langle\phi,\phi\rangle_{L^2(\lambda_D )}\\ =\mu\|\psi\|_{L^2(\lambda_D )^d}^2+\|\phi\|_{L^2(\lambda_D )}^2 \geq \min\big\{1,(2r)^{-1}\big\}\left\|\begin{pmatrix}
                        \phi \\ \psi
                       \end{pmatrix}\right\|_{L^2(\lambda_D )^{d+1}}^2,
\end{multline}
which yields \eqref{eq:condition_M}. Using Theorem \ref{thm:spdes}, we have thus proven the following. 

\begin{cor}\label{cor:SHE}
  With the notations from the beginning of this section, assume that $\sigma\colon L^2(\lambda_D )\to L_2(G,L^2(\lambda_D ))$ satisfies  
  \[ \|\sigma(u)-\sigma(v)\|_{L_2}\leq L\|u-v\|_{H} \]
  for all $u,v\in H$ and some $L\geq0$.

  Then there exists $\nu_1\geq0$ such that for all $\nu>\nu_1$, the equation 
  \begin{equation*}
 \overline{\left(\partial_0 \begin{pmatrix}
   0 & 0 \\ 0 & 1
 \end{pmatrix} + \begin{pmatrix}
   1 & 0 \\ 0 & 0
 \end{pmatrix}+\begin{pmatrix}
   0 & \div \\ \interior{\grad} & 0
 \end{pmatrix}\right)}\begin{pmatrix}
   u \\ q
 \end{pmatrix}=\begin{pmatrix}
   \int_0^{\cdot}\sigma(u)dW \\ 0
 \end{pmatrix},
 \end{equation*}
has a unique solution $(u,q)\in H_{\nu,0}(\R;P_W\otimes 1_{L^2(\lambda_D )^{d+1}})$, which is independent of $\nu$. (For a definition of $P_W$ one might recall Remark \ref{rem:filtisfilter}.)
\end{cor}

The solution theory is not limited to the case of partial differential operators with constant coefficients. The following remark shows how to invoke partial differential operators with variable coefficients.

 \begin{rem}\label{r:vc}
   Starting out with a deterministic bounded measurable matrix-valued coefficient function $a\colon D\to \mathbb{C}^{d\times d}$ being pointwise self-adjoint and uniformly strictly positive, that is, $\langle a(x)\xi,\xi\rangle\geq c\langle \xi,\xi\rangle$ for all $x\in D$, $\xi\in \Rd$ and some $c>0$, we consider the stochastic heat equation
  \[
     \partial_0 u - \div a\interior{\grad} u = \sigma(u)\dot W,
  \]
with the same vanishing boundary and initial data as above. Substituting $a^{-1} q = -\interior{\grad} \partial_0^{-1}u$,
we arrive at the system
\[ (\partial_0 M(\partial_0^{-1})+A)\begin{pmatrix} u\\ q\end{pmatrix} = \begin{pmatrix}\int_0^\cdot \sigma(u)dW \\0 \end{pmatrix}, \]
with the same $A$ as in the constant coefficient case and 
\[ M(z)=\begin{pmatrix}
   0 & 0 \\ 0 & a^{-1}
 \end{pmatrix} + z\begin{pmatrix}
   1 & 0 \\ 0 & 0
 \end{pmatrix}. \]
Under the conditions on $a$, we can show the existence and uniqueness of solutions using Theorem \ref{thm:spdes}. (Note that $a(x)=a(x)^*\geq c>0$ implies 
$a(x)^{-1}\geq c/\|a(x)\|^2$ in the sense of positive definiteness.)
\end{rem}

\subsubsection*{Connection to variational solutions}\label{su:cvs}
 We will compare our solution to the one defined in \cite[Definition 2.1]{sanzvuillermot} (with $g=0$ and $\phi=0$), see also \cite{prevotroeckner,rozowskii}. We understand the following notion as a variational/weak solution to the heat equation \eqref{eq:stoch_heat}:

\begin{defn}\label{def:variational}
  A predictable stochastic process $u$ supported on $[0,\infty)$ with values in $H^1_0(D)$ is called a \emph{variational solution to the stochastic heat equation} if 
  \[ \int_0^T\|u(t)\|^2_{H^1(D)}dt < \infty \]
  for all $0\leq T<\infty$ almost surely, has at most exponential growth in $T$ almost surely (with some exponential growth bound $\nu>0$), and for all $\eta\in H_0^1(D)$ the following equation holds almost surely for all $t\geq 0$ 
	\begin{equation}\label{eq:defvarsol} 
	  \langle u(t),\eta\rangle_{L^2(\lambda_D)} +\int_{0}^t \langle \grad u(\tau),\grad\eta\rangle_{L^2(\lambda_D )^{d}}d \tau = \sum_{k\in\N}\int_0^t \sqrt{\lambda_k} \langle \sigma(u(\tau))e_k,\eta\rangle_{L^2(\lambda_D )} dW_k(\tau),
	\end{equation}
  where $(e_k)_{k\in\N}$ is an orthogonal basis of $L^2(\lambda_D )$, $(\lambda_k)_{k\in\N}\in\ell_1(\N)$ is the sequence of eigenvalues of the covariance operator of $W$, and $(W_k)_{k\in\N}$ is a sequence of independent one-dimensional Brownian motions.
\end{defn}

In the next few lines, we will show that any variational solution in the sense of Definition \ref{def:variational} is a solution constructed in Corollary \ref{cor:SHE}. In order to avoid unnecessarily cluttered notation, we shall occasionally neglect referring to the real numbers in the notation of the vector-valued spaces to be studied in the following. For instance, for $H_{\nu,0}(\mathbb{R};H_0^1(D))$ we write $H_{\nu,0}(H_0^1(D))$ instead.

\begin{prop} Let $\nu>0$, $u\in H_{\nu,0}(\mathbb{R};H_0^1(D)\otimes L^2(\P))$ a variational solution to the stochastic heat equation. Then, $(u,-\partial_{0}^{-1}\interior{\grad} u)$ solves the equation for $(u,q)$ given in Corollary \ref{cor:SHE}.
\end{prop}
\begin{proof}
  Since $u$ is supported on $[0,\infty)$ only, we get, using Remark \ref{r:pif},
\[
   \int_0^t \interior{\grad} u(\tau) d\tau = \int_{-\infty}^t \interior{\grad} u(\tau) d\tau = \partial_{0}^{-1} \interior{\grad} u (t).
\]
Next, by Definition \ref{defn:stoch_int}, we obtain
\[
      \sum_{k\in\N}\int_0^t \sqrt{\lambda_k} \langle \sigma(u(\tau))e_k,\eta\rangle_{L^2(\lambda_D )} dW_k(\tau) = \langle \int_0^t   \sigma(u(\tau)) dW(\tau) , \eta \rangle_{L^2(\lambda_D )}.
\]
Hence, for all $v\in H_{\nu,0}(\mathbb{R}), \eta\in H_0^1(D)$, we obtain from \eqref{eq:defvarsol}
\[
   \langle u, v \eta \rangle_{H_{\nu,0}(L^2(\lambda_D ))} + \langle \partial_0^{-1} \interior{\grad} u, v \interior{\grad} \eta\rangle_{H_{\nu,0}(L^2(\lambda_D )^d)} = \langle \int_0^{(\cdot)} \sigma(u(\tau)) dW(\tau), v \eta\rangle_{H_{\nu,0}(L^2(\lambda_D ))}.
\]
In consequence, by linearity and continuity, we obtain for all $\phi\in H_{\nu,0}(\mathbb{R};H_0^1(D))$
\[
   \langle u, \phi \rangle_{H_{\nu,0}(L^2(\lambda_D ))} + \langle \partial_0^{-1} \interior{\grad} u,  \interior{\grad} \phi\rangle_{H_{\nu,0}(L^2(\lambda_D )^d)} = \langle \int_0^{(\cdot)} \sigma(u(\tau)) dW(\tau), \phi\rangle_{H_{\nu,0}(L^2(\lambda_D ))}.
\]
Substituting $q\coloneqq -\partial_0^{-1} \interior{\grad} u$, we obtain $\partial_0 q =- \interior{\grad} u$. Hence, $(u,q)$ solves the equation in Corollary \ref{cor:SHE} (even without the closure bar).
\end{proof}

For the reverse direction, we need an additional regularity assumption. Before commenting on this, we shall derive an equality, which is almost the one in \eqref{eq:defvarsol}. We need the following prerequisite of abstract nature.

\begin{lem}\label{l:dcp} Let $H_0$, $H_1$ be Hilbert spaces, $\nu>0$, $C\colon \dom(C)\subseteq H_0 \to H_1$ densely defined, closed. Then $\overline{\partial_0^{-1}C} = C\partial_0^{-1}$.
\end{lem}
\begin{proof} The operator $\partial_0^{-1}$ is continuous from $H_{\nu,0}(H_0)$ into itself and the operator $C$ is closed. Hence, $\overline{\partial_0^{-1}C}\subseteq C\partial_0^{-1}$. On the other hand, note that $(1+\eps C^*C)^{-1}\to 1$ and $(1+\eps CC^*)^{-1}\to 1$ in the strong operator topology as $\eps\to 0$. Thus, for $u\in \dom(C\partial_0^{-1})$ we let $u_\eps\coloneqq (1+\eps C^*C)^{-1}u$  and get $u_\eps\in \dom(C)=\dom(\partial_0^{-1} C)$. Moreover, $C (1+\eps C^*C)^{-1}$ is a continuous operator and $(1+\eps CC^*)^{-1} C\subseteq C (1+\eps C^*C)^{-1}$. Hence, for $\eps>0$
\[
   \partial_0^{-1} C u_\eps = \partial_0^{-1} C (1+\eps C^*C)^{-1} u = C (1+\eps C^*C)^{-1}\partial_0^{-1} u =(1+\eps CC^*)^{-1}C\partial_0^{-1} u.
\]
Letting $\eps\to0$ in the latter equality, we obtain $u\in \dom(\overline{\partial_0^{-1}C})$ and $\overline{\partial_0^{-1}C} u =C\partial_0^{-1} u$, which yields the assertion.
\end{proof}

\begin{thm}\label{t:solvar} Let $(u,q)\in H_{\nu,0}(\mathbb{R};(L^2(\lambda_D )\times L^2(\lambda_D )^d)\otimes L^2(\P))$ be a predictable process solving the equation in Corollary \ref{cor:SHE}. Then $q\in \dom(\div)$ and $u\in \dom(\interior{\grad}\partial_0^{-1})$, $u=-\interior{\grad}\partial_0^{-1}q$ and 
\begin{equation}\label{eq:defvarsol2} 
	  \langle u(\cdot),\eta\rangle_{L^2(\lambda_D ))} +\langle \grad \int_{0}^{(\cdot)} u(\tau)d \tau,\grad\eta\rangle_{L^2(\lambda_D )^{d}} = \sum_{k\in\N}\int_0^{(\cdot)} \sqrt{\lambda_k} \langle \sigma(u(\tau))e_k,\eta\rangle_{L^2(\lambda_D )} dW_k(\tau),
\end{equation}almost surely.
\end{thm}
\begin{proof}
  By causality and the fact that $W=0$ for negative times, we infer $(u,q)$ is supported on $[0,\infty)$ only. Moreover, by Remark \ref{rem:integratedPDE} ((a) and (b)), we obtain that 
  \[(u_\eps,q_\eps)\coloneqq \big((1+\eps\partial_0)^{-1}u,(1+\eps\partial_0)^{-1}q\big)\in \dom(\partial_0)\cap \dom\Big(\begin{pmatrix} 0 & \div \\ \interior{\grad} & 0 \end{pmatrix}\Big).
  \]
  Furthermore, we have
\begin{equation}\label{eq:SHEreg}
 \left(\partial_0 \begin{pmatrix}
   0 & 0 \\ 0 & 1
 \end{pmatrix} + \begin{pmatrix}
   1 & 0 \\ 0 & 0
 \end{pmatrix}+\begin{pmatrix}
   0 & \div \\ \interior{\grad} & 0
 \end{pmatrix}\right)\begin{pmatrix}
   u_\eps \\ q_\eps
 \end{pmatrix}=\begin{pmatrix}
   (1+\eps\partial_0)^{-1}\int_0^{\cdot}\sigma(u)dW \\ 0
 \end{pmatrix}.
\end{equation}     
Hence, the first line of the latter equality yields
\[
   u_\eps + \div q_\eps = (1+\eps\partial_0)^{-1}\int_0^{\cdot}\sigma(u)dW.
\]
Thus, using $(1+\eps\partial_0)^{-1}\to 1$ as $\eps\to0$ in the strong operator topology, we obtain by the closedness of $\div$ that
\begin{equation}\label{eq:diveq}
   q\in \dom(\div) \text{ and } \div q = -u+\int_0^{\cdot}\sigma(u)dW.
\end{equation}
Next, the second line of \eqref{eq:SHEreg} reads
\[
   \partial_0 q_\eps +\interior{\grad} u_\eps = 0\text{ or }  q_\eps +\partial_0^{-1}\interior{\grad} u_\eps = 0.
\]
Thus, by Lemma \ref{l:dcp}, we obtain as $\eps\to 0$,
\[
   q = -\interior{\grad}\partial_0^{-1} u.
\]
Therefore, from \eqref{eq:diveq} we read off
\[
   u- \div \interior{\grad}\partial_0^{-1} u = \int_0^{\cdot}\sigma(u)dW.
\]
Thus, testing the latter equality with $\eta\in H_0^1(D)=\dom(\interior{\grad})$, and using that 
\[
  -\langle \div \interior{\grad}\partial_0^{-1} u, \eta \rangle_{L^2(\lambda_D )} =\langle \interior{\grad}\partial_0^{-1} u, \interior{\grad}\eta \rangle_{L^2(\lambda_D )}
\]
we infer the asserted equality.
\end{proof}

\begin{cor} In the situation of Theorem \ref{t:solvar}, we additionally assume that $u\in H_{\nu,0}(H_0^1(D)\otimes L^2(\P))$. Then $u$ is a solution in the sense of Definition \ref{def:variational}. 
\end{cor}

\subsection{Stochastic wave equation}\label{sec:SWE}
Similarly to the treatment of the stochastic heat equation in the previous section, we show now how to reformulate the stochastic wave equation into a first order system and then prove the existence and uniqueness of solution. This equation has been treated in \cite{walsh, dalang} with a random-field approach and for instance in \cite[Example 5.8, Section 13.21]{dapratozabczyk} with a semi-group approach. Consider the following equation
\begin{align}\label{eq:stoch_wave}
   \partial_0^2u - \Delta u & = \sigma(u)\dot W, \notag\\
   u(0)=0, \partial_0u(0)=0, u|_{\partial D} & =0.
\end{align}
As in the previous section, we first apply the operator $\partial_0^{-1}$ to \eqref{eq:stoch_wave}, write $\Delta = \div\interior{\grad}$, and finally define $v\coloneqq \interior{\grad} \partial_0^{-1} u$. With these manipulations, we arrive at the following first-order system
\begin{equation}\label{eq:wave1}
   \left(\partial_0 \begin{pmatrix}
   1 & 0 \\ 0 & 1
 \end{pmatrix} -\begin{pmatrix}
   0 & \div \\ \interior{\grad} & 0
 \end{pmatrix}\right)\begin{pmatrix}
   u \\ v
 \end{pmatrix}=\begin{pmatrix}
   \int_0^{\cdot}\sigma(u)dW \\ 0
 \end{pmatrix},
\end{equation}
which we think of as the appropriate formulation for the stochastic wave equation. Now we can show, with $\sigma\colon L^2(\lambda_D )\to L_2(G,L^2(\lambda_D ))$ Lipschitz continuous, the existence and uniqueness of a solution to \eqref{eq:wave1} with the help of Theorem \ref{thm:spdes}. The only thing to be verified is that 
 \[
   M(z)\coloneqq \begin{pmatrix}
   1 & 0 \\ 0 & 1
 \end{pmatrix}
 \]
satisfies condition \eqref{eq:condition_M} for all $z\in B(r,r)$, for some $r>0$. This, however, is easy (see also the computation in \eqref{eq:M(z)}). Thus, we just obtained the following:
\begin{cor} There is $\nu_0>0$ such that for all $\nu\geq\nu_0$, there is a unique $(u,v)\in H_{\nu,0}(\mathbb{R};P_W\otimes 1_{L^2(\lambda_D )^{d+1}})$ such that
\[
  \overline{\left(\partial_0 \begin{pmatrix}
   1 & 0 \\ 0 & 1
 \end{pmatrix} -\begin{pmatrix}
   0 & \div \\ \interior{\grad} & 0
 \end{pmatrix}\right)}\begin{pmatrix}
   u \\ v
 \end{pmatrix}=\begin{pmatrix}
   \int_0^{\cdot}\sigma(u)dW \\ 0
 \end{pmatrix}.
\]The solution is independent of $\nu$.
\end{cor}

We emphasize that the way of writing the stochastic wave equation into a first-order-in-time system is not unique. Indeed, a more familiar way is to set $w\coloneqq -\Delta\partial_0^{-1}u$. With this we arrive at the following system
\begin{equation}\label{eq:wave2}
   \left(\partial_0 \begin{pmatrix}
  1 & 0 \\ 0 & 1
\end{pmatrix} +\begin{pmatrix}
  0 & 1\\ \Delta & 0
\end{pmatrix}\right)\begin{pmatrix}
  u \\ w
\end{pmatrix}=\begin{pmatrix}
  \int_0^{\cdot}\sigma(u)dW \\ 0
\end{pmatrix}.     
\end{equation}
The latter system is essentially the same as the system in \eqref{eq:wave1}, see \cite[p.~16/17]{Picard2013} for the mathematically rigorous statement. However, the spatial Hilbert spaces differ from one another: in \eqref{eq:wave1} the spatial Hilbert space is $H=L^2(\lambda_D )\oplus L^2(\lambda_D )^d$, and in \eqref{eq:wave2} it coincides with $H=H_0^1(D)\oplus L^2(\lambda_D )$. The domains of the two spatial partial differential operators
\[ \begin{pmatrix}
   0 & \div \\ \interior{\grad} & 0
 \end{pmatrix}\quad\text{and}\quad\begin{pmatrix} 
   0 & 1\\ \Delta & 0
 \end{pmatrix}\]
are $\dom(\interior{\grad})\oplus \dom(\div)= H_0^1(D)\oplus \dom(\div)$ and $\dom(\Delta)\oplus H_0^1(D)$, respectively, where $\dom(\Delta)=\dom(\div\interior{\grad})$. However, the solvability of one system implies the solvability of the other one. In any case, for bounded $D$, endowing $H_0^1(D)$ with the scalar product induced by $(u,v)\mapsto \langle \grad u,\grad v\rangle$, it can be shown that 
\[
   \quad\begin{pmatrix} 
   0 & 1\\ \Delta & 0
 \end{pmatrix} \colon \dom(\div\interior{\grad})\oplus H_0^1(D)\subseteq H_0^1(D)\oplus L^2(\lambda_D )\to H_0^1(D)\oplus L^2(\lambda_D ), (u,v)\mapsto (v,\Delta u)
\]
is skew-self-adjoint. For the latter assertion, it is sufficient to note the following proposition:
\begin{lem}\label{l:ske} Assume that $D\subseteq \mathbb{R}^d$ is bounded. Let $C\colon \dom(\div\interior{\grad})\subseteq H_0^1(D) \to L^2(\lambda_D )$ with $Cu=\Delta u$. Then, for $C^*\colon \dom(C^*)\subseteq L^2(\lambda_D )\to H_0^1(D)$ we have
\[
  C^*v= - v\text{ for all } v\in \dom(C^*)=H_0^1(D),
\]
where $H_0^1(D)$ is endowed with the scalar product $(u,v)\mapsto \langle \grad u,\grad v\rangle$. 
\end{lem}
\begin{proof} Let $v\in L^2(\lambda_D )$ and $f\in H_0^1(D)$. Then we compute
\begin{align*}
  v\in \dom(C^*), C^*v = f & \iff \forall \phi\in \dom(C)\colon \langle C\phi , v \rangle_{L^2(\lambda_D )}=\langle \phi,f\rangle_{H_0^1(D)} 
  \\ & \iff \forall \phi\in \dom(\div\interior{\grad})\colon \langle \div\interior{\grad}\phi , v \rangle_{L^2(\lambda_D )}=\langle \interior{\grad}\phi,\interior{\grad}f\rangle_{L^2(\lambda_D )}
  \\ & \iff \forall \phi\in \dom(\div\interior{\grad})\colon \langle \div\interior{\grad}\phi , v \rangle_{L^2(\lambda_D )}=-\langle\div \interior{\grad}\phi,f\rangle_{L^2(\lambda_D )}
  \\ & \iff \forall \phi\in \dom(\div\interior{\grad})\colon \langle \div\interior{\grad}\phi , v+f \rangle_{L^2(\lambda_D )}=0.
\end{align*}
But, $\div\interior{\grad} \colon \dom(\div\interior{\grad})\subseteq L^2(\lambda_D )\to L^2(\lambda_D )$ is continuously invertible. In particular, $\div\interior{\grad}$ is onto. Hence,
\[
   \forall \phi\in \dom(\div\interior{\grad})\colon \langle \div\interior{\grad}\phi , v+f \rangle_{L^2(\lambda_D )}=0 \iff v=-f\in H_0^1(D).
\]
The assertion follows. 
\end{proof}

Hence, with Lemma \ref{l:ske} in mind, in either formulation -- \eqref{eq:wave1} or \eqref{eq:wave2} -- our solution theory, Theorem \ref{thm:spdes}, applies. We shall also note that the functional analytic framework provided serves to treat deterministic variable coefficients $a\colon D\to \mathbb{C}^{d\times d}$ satisfying the same assumptions as in Remark \ref{r:vc} and to treat the corresponding wave equation
\[
   (\partial_0^2-\div a \interior{\grad}) u = \sigma(u)\dot W
\]
or, rather,
\[
  \overline{\left(\partial_0 \begin{pmatrix}
   1 & 0 \\ 0 & a^{-1}
 \end{pmatrix} -\begin{pmatrix}
   0 & \div \\ \interior{\grad} & 0
 \end{pmatrix}\right)}\begin{pmatrix}
   u \\ v
 \end{pmatrix}=\begin{pmatrix}
   \int_0^{\cdot}\sigma(u)dW \\ 0
 \end{pmatrix}.
\]

\subsubsection*{Connection to mild solutions}\label{su:cms}

Next, we will comment on the relationship of the solution obtained for \eqref{eq:wave2} to a more classical way of deriving the solution by means of $C_0$-semi-groups: On the bounded, open $D\subseteq \mathbb{R}^d$ consider the classical reformulation of the stochastic wave equation as a first-order system
\begin{equation}\label{eq:systemwaveDPZ}
  \left(\begin{pmatrix} \partial_0 & 0 \\ 0 & \partial_0 
\end{pmatrix} - \begin{pmatrix} 0 & 1 \\ \Delta & 0 
\end{pmatrix}\right)\begin{pmatrix} u \\ v \end{pmatrix} = 
\begin{pmatrix} 0 \\ \sigma(u)\dot W \end{pmatrix},
\end{equation}
with zero initial conditions and homogeneous Dirichlet boundary conditions, see \cite[Example 5.8]{dapratozabczyk} for this reformulation. So $\Delta=\div\interior{\grad}$ with a suitable domain. The solution to \eqref{eq:systemwaveDPZ} can be computed using the semi-group approach in \cite{dapratozabczyk} to be
\begin{align}\label{eq:SGu}
  \begin{pmatrix} u(t) \\ v(t) \end{pmatrix}
  & = \int_0^t \bfS(t-s)\begin{pmatrix} 0 \\ 
\sigma(u(s))dW(s)\end{pmatrix} \notag\\
  & = \begin{pmatrix} \int_0^t 
(-\Delta)^{-1/2}\sin((-\Delta)^{1/2}(t-s)\sigma(u(s))dW(s) \\ \int_0^t 
\cos((-\Delta)^{1/2}(t-s))\sigma(u(s))dW(s) \end{pmatrix},
\end{align}
where $\bfS(t)$ is the semi-group defined by
\[ \bfS(t) = \begin{pmatrix} \cos((-\Delta)^{1/2}t) & 
(-\Delta)^{-1/2}\sin((-\Delta)^{1/2}t) \\ 
-(-\Delta)^{1/2}\sin((-\Delta)^{1/2}t) & \cos((-\Delta)^{1/2}t) 
\end{pmatrix}. \]

However, in \eqref{eq:wave2}, we have arrived at a different reformulation as a first-order system, given by
\begin{equation}\label{eq:systemwave}
  \overline{\left(\begin{pmatrix} 
		\partial_0 & 0 \\ 0 & \partial_0 
	\end{pmatrix} 
	+ 
	\begin{pmatrix} 
		0 & 1	 \\ \Delta & 0 
	\end{pmatrix}\right)}
	\begin{pmatrix} 
		u \\ w 
	\end{pmatrix}
	= 
	\begin{pmatrix} 
		\int_0^\cdot \sigma(u(r))dW(r) \\ 0 
	\end{pmatrix}.
\end{equation}
Our aim in this section will be to establish the following result.

\begin{thm}\label{t:sgev} Let $(u,v)$ satisfy \eqref{eq:SGu}. Then $(u,w)$ solves \eqref{eq:systemwave} with 
\begin{equation}\label{eq:w}
  w = \int_0^\cdot \sigma(u(r))dW(r) - v.
\end{equation}
\end{thm}

For this, we need some elementary prerequisites:
\begin{lem}\label{l:spf} Let $r,t\in\mathbb{R}$. Then the following statements hold.

(a) For any $\zeta\in \mathbb{R}_{>0}$, we have
\[
   \int_r^t\zeta^{1/2}\sin(\zeta^{1/2}(s-r))ds=1-\cos(\zeta^{1/2}(t-r))
\]
 and 
 \[
    \int_r^t \cos(\zeta^{1/2}(s-r))ds=\zeta^{-1/2}\sin(\zeta^{1/2}(t-r)).
 \]
 
(b) For all $\phi\in \dom(\Delta)=\dom(\div\interior{\grad})$ we have
\[
  \int_r^t(-\Delta)(-\Delta)^{-1/2}\sin((-\Delta)^{1/2}(s-r))\phi ds=(I-\cos((-\Delta)^{1/2}(t-r)))\phi 
\]
and 
\[
  \int_r^t\cos((-\Delta)^{1/2}(s-r))\phi ds=(-\Delta)^{-1/2}\sin((-\Delta)^{1/2}(t-r))\phi.
\]
\end{lem}
\begin{proof} The equations in (a) can be verified immediately. In order to settle (b), we use the spectral theorem for the (strictly positive definite) Dirichlet--Laplace operator $-\Delta$ on the underlying open and bounded set $D$. Hence, (b) is a consequence of (a) by Fubini's theorem.
\end{proof}

Next, we proceed to a proof of the main result in this section.

\begin{proof}[Proof of Theorem \ref{t:sgev}] Using \eqref{eq:w} together with the second line of \eqref{eq:SGu}, we obtain for  $\phi\in \dom(\Delta)$
\begin{align*}
     &\langle w(t),\phi\rangle_{L^2(\lambda_D )} \\
     & =   \Big\langle \int_0^t  \bigg(I-\cos((-\Delta)^{1/2}(t-r))\bigg)\sigma(u(r)) dW(r), \phi\Big\rangle_{L^2(\lambda_D )}\\
     & =   \sum_{k\in \mathbb{N}} \lambda_k^{1/2}\int_0^t \langle \bigg(I-\cos((-\Delta)^{1/2}(t-r))\bigg)\sigma(u(r))e_k, \phi\rangle_{L^2(\lambda_D )} dW_k(r) \\
     & =   \sum_{k\in \mathbb{N}} \lambda_k^{1/2}\int_0^t \langle \sigma(u(r))e_k,\bigg(I-\cos((-\Delta)^{1/2}(t-r))\bigg) \phi\rangle_{L^2(\lambda_D )} dW_k(r).
\end{align*}     
     With Lemma \ref{l:spf}, we further obtain
     \begin{align*}
     &\langle w(t),\phi\rangle_{L^2(\lambda_D )} \\     
     & =   \sum_{k\in \mathbb{N}} \lambda_k^{1/2}\int_0^t \langle \sigma(u(r))e_k,\bigg(\int_r^t (-\Delta)(-\Delta)^{-1/2}\sin((-\Delta)^{1/2}(s-r))\phi ds\bigg) \rangle_{L^2(\lambda_D )} dW_k(r) \\
     & =   \sum_{k\in \mathbb{N}} \lambda_k^{1/2}\int_0^t \langle\bigg(\int_r^t (-\Delta)^{-1/2}\sin((-\Delta)^{1/2}(s-r)) ds\bigg) \sigma(u(r))e_k,-\Delta \phi\rangle_{L^2(\lambda_D )} dW_k(r) \\
     & =  \Big\langle \sum_{k\in \mathbb{N}} \lambda_k^{1/2}\int_0^t \bigg(\int_r^t (-\Delta)^{-1/2}\sin((-\Delta)^{1/2}(s-r)) ds\bigg) \sigma(u(r))e_k dW_k(r),-\Delta \phi\Big\rangle_{L^2(\lambda_D )} \\
     & =  \Big\langle \int_0^t \bigg(\int_r^t (-\Delta)^{-1/2}\sin((-\Delta)^{1/2}(s-r)) ds\bigg) \sigma(u(r))dW(r),-\Delta \phi\Big\rangle_{L^2(\lambda_D )} \\
     & =  \Big\langle \int_0^t \int_0^s 
(-\Delta)^{-1/2}\sin((-\Delta)^{1/2}(s-r))\sigma(u(r))dW(r)ds,-\Delta\phi\Big\rangle_{L^2(\lambda_D )} \\
     & =  \Big\langle \bigg(\partial_0^{-1} \int_0^\cdot 
(-\Delta)^{-1/2}\sin((-\Delta)^{1/2}(\cdot-r))\sigma(u(r))dW(r)\bigg)(t),-\Delta\phi\Big\rangle_{L^2(\lambda_D )}
     \\ & = \langle \partial_0^{-1}u (t), -\Delta \phi\rangle_{L^2(\lambda_D )},
\end{align*}
where in the last equality we used the first line of \eqref{eq:SGu}. We read off $\partial_0^{-1}u \in \dom(\Delta)$ and
\[
   w=-\Delta\partial_0^{-1}u
\]
Moreover, we compute with \eqref{eq:SGu} and Lemma \ref{l:spf},
\begin{align*}
   \partial_0^{-1} v(t) & = \int_0^t \int_0^s 
\cos((-\Delta)^{1/2}(s-r))\sigma(u(r))dW(r) d s
\\ & = \int_0^t \int_r^t  \cos((-\Delta)^{1/2}(s-r)) ds \sigma(u(r))dW(r)
\\ & = \int_0^t (-\Delta)^{-1/2}\sin((-\Delta)^{1/2}(t-r)) \sigma(u(r))dW(r)
\\ & = u(t).
\end{align*} Therefore, together with \eqref{eq:w}, we get
\begin{align*}
 \partial_0 u + w & = v+ w = \int_0^\cdot \sigma(u(r))dW(r),\\
 w + \Delta \partial_0^{-1} u & = 0.
\end{align*}
So, again by multiplying both these equations with $(1+\eps\partial_0)^{-1}$ and setting $u_\eps\coloneqq (1+\eps\partial_0)^{-1}u$ as well as $w_\eps\coloneqq (1+\eps\partial_0)^{-1}w$, we obtain
\[
   \left(\partial_0 \begin{pmatrix} 1 & 0 \\ 0 & 1 \end{pmatrix} + \begin{pmatrix} 0 & 1 \\ \Delta & 0 \end{pmatrix}\right)\begin{pmatrix}
                                                                                                                            u_\eps\\ w_\eps 
                                                                                                                           \end{pmatrix} = \begin{pmatrix}
                                                                                                                            (1+\eps\partial_0)^{-1}\int_0^\cdot \sigma(u(r))dW(r)\\ 0
                                                                                                                           \end{pmatrix}.
\]
Hence, by letting $\eps\to 0$, we obtain the assertion.
\end{proof}

\subsection{Stochastic Schr\"odinger equation with additive noise}
In this section we treat the stochastic Schr\"odinger equation on an open set $D\subseteq \Rd$, see for instance \cite[Chapter 2]{barchielli}. It can be formulated as
\[
   \partial_0 u -\ii\Delta u = b(u)+\partial_0 X,\quad u(0)=0,
\]
with appropriate boundary conditions such that $\Delta$ becomes a self-adjoint operator (recall that then $\ii \Delta$ is skew-selfadjoint) and
\[ b\colon H_{\nu,-1}(\R;L^2(\lambda_D )\otimes L^2(\P))\to H_{\nu,-1}(\R;L^2(\lambda_D )\otimes L^2(\P))\]
being Lipschitz continuous with Lipschitz constant less than $\nu$. We assume that $\partial_0 X$ is the derivative of a stochastic process as discussed in Section \ref{sec:additive}. Then the stochastic Schr\"odinger equation is well-posed according to Theorem \ref{thm:solth_add}.

\subsection{Stochastic Maxwell Equations}
Before discussing the stochastic Maxwell equations, we need to introduce some vector-analytic operators. In the whole section let $D\subseteq \R^3$ be open. 

\begin{defn}
We define
\begin{align*}
   \curl_c \colon C_c^\infty(D)^3\subseteq L^2(\lambda_D )^3 & \to L^2(\lambda_D )^3,\\
						\begin{pmatrix}\phi_1\\ \phi_2 \\ \phi_3\end{pmatrix} & \mapsto 
						\begin{pmatrix}0 & -\partial_3 &  \partial_2 \\
						              \partial_3 & 0 &  -\partial_1 \\ 
						              -\partial_2 & \partial_1 & 0 \end{pmatrix}\begin{pmatrix}\phi_1\\ \phi_2 \\ \phi_3\end{pmatrix},
\end{align*}
where $\partial_1,\partial_2,\partial_3$ are the partial derivatives with respect to the first, second and third spatial variable, respectively. Let $\curl\coloneqq \curl_c^*$ and $\interior\curl\coloneqq \curl^*$.   
\end{defn}

We introduce the linear operators $\epsilon,\mu,\zeta\in L(L^2(\lambda_D )^3)$, modeling the respective material coefficients dielectricity, magnetic permeability and electric conductivity, with the following additional properties
\begin{itemize}
  \item $\epsilon$ is self-adjoint and positive definite, $\epsilon^*=\epsilon\geq 0$,
  \item $\mu$ is self-adjoint, $\mu^*=\mu$, 
  \item both the operators $\mu$ and $\nu \epsilon+\Re \zeta$ are strictly positive definite if $\nu>0$ is chosen large enough.
\end{itemize}

Then Maxwell's equations can be written in the form
\[
   \left(\partial_0 \begin{pmatrix}
                \epsilon & 0 \\ 0 & \mu 
              \end{pmatrix} + \begin{pmatrix}
                \zeta & 0 \\ 0 & 0
              \end{pmatrix} + \begin{pmatrix}
                0& -\curl \\ \interior\curl & 0 
              \end{pmatrix}\right) \begin{pmatrix}
              E \\ H\end{pmatrix} = \begin{pmatrix}
              J \\ 0\end{pmatrix}.
\]
This first-order system is well-posed in solving for $(E,H)\in H_{\nu,0}(\R;L^2(\lambda_D )^6)$, where the quantity $J\in H_{\nu,0}(\R;L^2(\lambda_D )^3)$, the external currents, is a given right-hand side. Indeed, this follows from our deterministic solution theory in Theorem \ref{thm:solth}, see \cite[Section 3.1.12.4]{picardbook} for a detailed treatment.  Hence, incorporating stochastic integrals in the Maxwell equations leads to 
\[
   \left(\partial_0 \begin{pmatrix}
                \epsilon & 0 \\ 0 & \mu 
              \end{pmatrix} + \begin{pmatrix}
                \zeta & 0 \\ 0 & 0 
              \end{pmatrix} + \begin{pmatrix}
                0& -\curl \\ \interior\curl & 0 
              \end{pmatrix}\right) \begin{pmatrix}
              E \\ H\end{pmatrix} = \begin{pmatrix}
              \int_0^{(\cdot)}\sigma(E,H)dW +J\\ 0\end{pmatrix},
\]
which in turn is well-posed by Theorem \ref{thm:spdes}.

\begin{rem}
 Note that the Maxwell equations with multiplicative noise have not been -- to the best of our knowledge -- discussed yet in the literature. The above formulation of this particular reformulation is in fact a possible way to understand the `stochastic Maxwell equations with multiplicative noise'. Stochastic Maxwell equations with additive noise have, however, been discussed in the literature, see for instance \cite{CJZ16,HSY10,JLZ14}. A solution theory for this line of problem can again be found in Section \ref{sec:additive}.
\end{rem}

\subsection{SPDEs with fractional time derivatives}
Due to the generality of our ansatz with respect to the freedom in the operator coefficient $M(\partial_0^{-1})$, we may also treat stochastic partial differential equations with fractional time derivatives. As an instant, let us consider the following super-diffusion equation for $\alpha\in (0,1)$:
\begin{align*}
  \partial_0^{1+\alpha}u - \Delta u & = \sigma(u)\dot W,
\end{align*}
subject to zero initial and, for instance, homogeneous Neumann boundary conditions in an open set $D\subseteq \Rd$. Note that we can incorporate these homogeneous boundary conditions in the formulation of the abstract setting (without regards to the smoothness of the boundary of $D$) in the way that $\Delta\coloneqq \interior{\div}\grad$, where $\interior{\div}$ carries the homogeneous Neumann boundary conditions. As in the previous sections, we define an auxiliary unknown $v\coloneqq -\grad \partial_0^{-1}u$ and get the following system
\[
  \begin{pmatrix} 
     \partial_0^{\alpha} & 0 \\ 0 & \partial_0
  \end{pmatrix}
  \begin{pmatrix} 
     u \\ v
  \end{pmatrix}
  +
  \begin{pmatrix} 
     0 & \interior{\div} \\ \grad & 0
  \end{pmatrix}
  \begin{pmatrix} 
     u \\ v
  \end{pmatrix}
  =
  \begin{pmatrix} 
     \int_0^{(\cdot)}\sigma(u)dW \\ 0
  \end{pmatrix}
\]
as the appropriate formulation for the stochastic super-diffusion equation discussed above. Recall that the part with the time derivative is given by $\partial_0M(\partial_0^{-1})$, where here $M$ is given by
\[
   M(z) \coloneqq \begin{pmatrix}
                   z^{\alpha-1} & 0 \\ 0 & 1
                  \end{pmatrix}.
\]
It can be shown that this $M$ satisfies the condition of strict positive definiteness for all $z\in B(r,r)$ for all $r>0$ in \eqref{eq:condition_M}, see \cite[Lemma 2.1 or Theorem 3.5]{Drrerfrac}. Hence, Theorem \ref{thm:spdes} is applicable and well-posedness is established.

\begin{rem}
  Of course one can think of more complicated equations containing fractional (time) derivatives. For other possible equations containing fractional (time) derivatives, we refer to \cite{homfrac,Drrerfrac} and the references therein. In order to limit the extend of this exposition, we postpone a more detailed survey of fractional stochastic partial differential equation to future work.
\end{rem}

\subsection{An equation of mixed type}\label{s:mix}

In the following we demonstrate the usefulness of the approach presented by applying our main theorem to an equation of mixed type. We refer to the textbooks \cite{B64,R90} as standard references for equations of mixed type. In these references, the authors also sketch a link to real world applications such as problems in fluid or gas dynamics. Furthermore, we refer to the eddy current approximation in electrodynamics, which forms a mixed type problem changing its type from hyperbolic to parabolic on different space-time domains, see \cite{PP16}. In order to provide a simple example, we discuss the following model problem. For this, let ${D}\subseteq \mathbb{R}^d$ be an open set, ${D}_{\textnormal{e}},{D}_{\textnormal{p}},{D}_{\textnormal{h}}\subseteq {D}$ pairwise disjoint and measurable. Assume that ${D}_{\textnormal{e}}\cup {D}_{\textnormal{p}}\cup {D}_{\textnormal{h}}={D}$. On $H_{\nu,0}(\mathbb{R};L^2(\lambda_{D})\oplus L^2(\P))$ consider the equation of mixed type
\begin{equation}\label{eq:mt}
   \left(\partial_{0} \begin{pmatrix}
                 \1_{{D}_{\textnormal{h}}} & 0 \\\ 0 & \1_{{D}_{\textnormal{p}}}+\1_{{D}_{\textnormal{h}}} 
                \end{pmatrix} + \begin{pmatrix}
                 \1_{{D}_{\textnormal{p}}}+\1_{{D}_{\textnormal{e}}} & 0 \\\ 0 & \1_{{D}_{\textnormal{p}}}+\1_{{D}_{\textnormal{e}}}
                \end{pmatrix}- \begin{pmatrix}
                 0 & \div \\ \interior{\grad} & 0 \end{pmatrix}\right)\begin{pmatrix} u \\ q \end{pmatrix} =\begin{pmatrix} F \\ 0 \end{pmatrix}.
\end{equation}
If $F=\int_0^{(\cdot)}\sigma(u)d W$ for some suitable $\sigma$ and a Wiener process $W$, we are in the position of applying Theorem \ref{thm:spdes}. Indeed, note that for all $\nu>0$ the operator family
\[
   M(z) = \begin{pmatrix}
                 \1_{{D}_{\textnormal{h}}} & 0 \\\ 0 & \1_{{D}_{\textnormal{p}}}+\1_{{D}_{\textnormal{h}}} 
                \end{pmatrix} +  z\begin{pmatrix}
                 \1_{{D}_{\textnormal{p}}}+\1_{{D}_{\textnormal{e}}} & 0 \\\ 0 & \1_{{D}_{\textnormal{p}}}+\1_{{D}_{\textnormal{e}}}\end{pmatrix}\quad (z\in B(r,r), r>1/(2\nu))
\]
satisfies the positive definiteness condition of Theorem \ref{thm:spdes}. Note that Equation \eqref{eq:mt} is indeed an equation of mixed type: On ${D}_\textnormal{h}$ the equation admits the form of the stochastic wave equation as in Section \ref{sec:SWE}. On ${D}_{\textnormal{p}}$, Equation \eqref{eq:mt} admits the form of the one in Section \ref{sec:SHE}, which is the stochastic heat equation (one has to put $q=\partial_0^{-1}\interior{\grad}$). The equation under consideration in this section is of elliptic type on the set ${D}_{\textnormal{e}}$. 

We note here that it is not needed to implement transmission conditions on the interfaces $\partial{D}_{\textnormal{e}}\cap {D}$, $\partial{D}_{\textnormal{p}}\cap {D}$, and $\partial{D}_{\textnormal{h}}\cap {D}$ additionally. In fact, the condition of $(u,q)$ being in the domain of $\overline{\left(\partial_{0} \begin{pmatrix}
                 \1_{{D}_{\textnormal{h}}} & 0 \\\ 0 & \1_{{D}_{\textnormal{p}}}+\1_{{D}_{\textnormal{h}}} 
                \end{pmatrix} + \begin{pmatrix}
                 \1_{{D}_{\textnormal{p}}}+\1_{{D}_{\textnormal{e}}} & 0 \\\ 0 & \1_{{D}_{\textnormal{p}}}+\1_{{D}_{\textnormal{e}}}
                \end{pmatrix}- \begin{pmatrix}
                 0 & \div \\ \interior{\grad} & 0 \end{pmatrix}\right)}$ is the appropriate realization of transmission conditions, see also the treatment of a mixed type problem in \cite[Remark 3.2]{Waurick2016_StH}. 
                 
It remains unclear of how to solve the equation in question with classical methods. In particular, if one is to use the semi-group approach, it is unclear of how to define an appropriate semi-group. 

\section{Conclusion}\label{s:con}

We presented an attempt for a unified solution theory for a class of stochastic partial differential equations. The concept is an adaption of the deterministic solution theory developed in \cite{PicPhy} and, thus, it applies to various physical phenomena. More precisely, we perturbed the deterministic equation by a stochastic right-hand side. This right-hand side turned out to be Lipschitz continuous since the solution operator of the deterministic PDE leaves -- thanks to causality -- predictable processes invariant.

For particular cases, we demonstrated that the solutions derived coincide with `variational solutions' or `mild solutions'. However, we emphasize that -- even in the deterministic setting -- the solution concept developed is different to the semi-group approach. On the one hand, even though the solution theory in Theorem \ref{thm:solth} may be extended to closed densely defined operators $A$ satisfying $\Re\langle A\phi,\phi\rangle, \Re\langle A^*\psi,\psi\rangle\geq 0$ for all $\phi\in \dom(A)$, $\psi\in \dom(A^*)$, the solution theory given by Theorem \ref{thm:solth} does not extend to all equations which are covered by semi-group-methods as the latter may be carried over to the Banach space case. Thus, there are equations that may be solved via the semi-group method, that cannot be solved with the approach presented here. On the other hand, there are also equations that are not covered by semi-groups, which nonetheless fall into the class of evolutionary equations, see, for instance, \cite{Waurick2016_StH} or \cite{Controlleti,Controlleti2,Waurick2014MMAS_Non,PTWW}. 

The main application of the present results maybe to derive a solution concept for (S)PDEs when the semi-group approach fails and the existence of a semi-group (sufficiently regular fundamental solution) cannot be shown. In particular, if one is confronted with equations of mixed type, see Section \ref{s:mix}, the present approach may be advanced. Further applications can be found in differential-algebraic system as in control theory, see \cite{Controlleti,Controlleti2}. Moreover, the current approach may open the doors for solution concepts for SPDEs, whilst imposing rather mild (if any) conditions on the regularity of the coefficients or the boundary of the underlying spatial domain.

\bibliographystyle{abbrv}

\end{document}